\documentclass[letter,final,twosided]{amsart}
\usepackage{amsfonts}
\usepackage{amscd}
\usepackage{amssymb}
\usepackage{amsmath}
\usepackage{amstext}
\usepackage{amsmath}
\usepackage{paralist}
\usepackage{graphicx}
\usepackage[english]{babel}
\usepackage{hyperref}

\newtheorem{theorem}{Theorem}[section]

\theoremstyle{definition}
\newtheorem{definition}[theorem]{Definition}
\newtheorem{remark}[theorem]{Remark}

\title[Discrete second order constrained Lagrangian systems] {Discrete
  second order constrained \\Lagrangian systems: first results}

\author[N. Borda, J. Fern\'andez and S. Grillo]{}



\email{bordan@ib.cnea.gov.ar}
\email{jfernand@ib.edu.ar}
\email{sergiog@cab.cnea.gov.ar}

\thanks{This research was partially supported by grants from the
  Universidad Nacional de Cuyo (UNCu), the Comisi{\'o}n Nacional de
  Energ{\'\i}a At{\'o}mica (CNEA) and the Universidad Nacional de La Plata
  (UNLP). N. B. was partially supported by fellowships from the
  Fundaci{\'o}n YPF and the Consejo Nacional de Investigaciones
  Cient{\'\i}ficas y T{\'e}cnicas (CONICET)}

\begin{document}

\bibliographystyle{amsplain}


\maketitle

\centerline{\scshape Nicol{\'a}s Borda$^{1,2,3}$, Javier Fern{\'a}ndez$^1$ and
  Sergio Grillo$^{1,3}$}

\medskip {\footnotesize 

  \centerline{$^1$Instituto Balseiro, Universidad Nacional de Cuyo --
    C.N.E.A.}  

  \centerline{ Av. Bustillo 9500, San Carlos de
    Bariloche, R8402AGP, Rep\'ublica Argentina} }

\medskip

{\footnotesize
  \centerline{$^2$Departamento de Matem{\'a}tica, Facultad de Ciencias Exactas,
    Universidad Nacional de La Plata} 
  \centerline{50 y 115, La Plata, Buenos Aires, 1900, Rep\'ublica
    Argentina} 
}

\medskip 

{\footnotesize \centerline{$^3$Consejo Nacional de Investigaciones
    Cient{\'\i}ficas y T{\'e}cnicas, Rep\'ublica Argentina} }


\begin{abstract}
  We briefly review the notion of second order constrained
  (continuous) system (SOCS) and then propose a discrete time
  counterpart of it, which we naturally call discrete second order
  constrained system (DSOCS). To illustrate and test numerically our
  model, we construct certain integrators that simulate the evolution
  of two mechanical systems: a particle moving in the plane with
  prescribed signed curvature, and the inertia wheel pendulum with a
  Lyapunov constraint. In addition, we prove a local existence and
  uniqueness result for trajectories of DSOCSs. As a first comparison
  of the underlying geometric structures, we study the symplectic
  behavior of both SOCSs and DSOCSs.
\end{abstract}


\section{Introduction}
\label{Sect:intro}

Discrete Variational Mechanics originated in the 60's, motivated by
the construction of variational numerical integrators for the
equations of motion of (continuous) mechanical systems. Since then,
significant progress has been made in the study of discrete time
versions of unconstrained systems and systems with holonomic
constraints. The advantage offered by the resulting integrators,
compared to other numerical methods, is that they take into account
the underlying geometric structure present in the mechanical problem
and, therefore, can be designed to respect, in some way, the momentum,
energy, or symplectic structure (see \cite{ar:marsden_west-discrete_mechanics_and_variational_integrators} and the
multiple references therein). The discrete dynamics in the more
general case of nonholonomic constraints\footnote{Classical references
  for (continuous) nonholonomic systems
  are~\cite{bo:dobronravov-the_fundamentals_of_mechanics_of_nonholonomic_systems,bo:neimark_fufaev-dynamics_of_nonholonomic_systems}. More recent ones
  are~\cite{bo:bloch-nonholonomic_mechanics_and_control,bo:cortes-non_holonomic}.}
was introduced more recently, in 2001, by J.  Cort\'es and
S. Mart\'{\i}nez in~\cite{ar:cortes_martinez-non_holonomic_integrators}. Nonholonomic integrators
have become of interest mainly because of their good performance in
numerical experiments (see, for
instance,~\cite{ar:mclachlan_perlmutter-integrators_for_nonholonomic_mechanical_systems,ar:campos_cendra_diaz_martin-discrete_lagrange_dalembert_poincare_equations_for_euler's_disk}). Still,
they are less understood theoretically than the preceding ones.

Even broader than the continuous nonholonomic case, we have mechanical
systems with higher order constraints, which have been studied
in~\cite{ar:cendra_grillo-lagrangian_systems_with_higher_order_constraints,ar:cendra_ibort_de_leon_martin-a_generalization_of_chetaevs_principle_for_a_class_of_higher_order_nonholonomic_constraints,ar:grillo-higher_order_constrained_hamiltonian_systems}. They are
Lagrangian systems with constraints involving higher order derivatives
of the position. They have been considered for describing some
simplified models of rolling viscoelastic bodies and systems with
friction~\cite{ar:cendra_ibort_de_leon_martin-a_generalization_of_chetaevs_principle_for_a_class_of_higher_order_nonholonomic_constraints,ar:cendra_grillo-lagrangian_systems_with_higher_order_constraints}. They have also
appeared in applications to the control of underactuated mechanical
systems~\cite{ar:grillo-higher_order_constrained_hamiltonian_systems,ar:grillo_maciel_perez-closed_loop_and_constrained_mechanical_systems,ar:grillo_marsden_nair-lyapunov_constraints_and_global_asymptotic_stabilization} (see
Section~\ref{Sect:example2}). Such applications consist in finding
constraints that ensure the desired behavior of the system under
consideration and, then, taking the related constraint force as the
control law (see also \cite{ar:marle-kinematic_and_geometric_constraints_servomechanism_and_control_of_mechanical_systems,ar:cendra_grillo-generalized_nonholonomic_mechanics_servomechanisms_and_related_brackets,ar:shiriaev_perram_canudasdewit-constructive_tool_for_orbital_stabilization_of_underactuated_nonlinear_systems}).  It is a
general fact that every control signal can be obtained by this
procedure using second order constraints~\cite{ar:grillo_maciel_perez-closed_loop_and_constrained_mechanical_systems}.
For example, for asymptotic stabilization of underactuated systems
Lyapunov constraints can be used (see~\cite{ar:grillo_marsden_nair-lyapunov_constraints_and_global_asymptotic_stabilization} and
Section~\ref{Sect:example2}).

It is worth remarking that the constraints
appearing in most of the interesting applications, like those
previously mentioned, involve, at most, second order derivatives,
\emph{i.e.} positions, velocities and accelerations. For this reason,
we will only consider systems with (at most) second order constraints
in this work.

The practical difficulty of solving the equations of motion of
(continuous) mechanical systems with nonholonomic constraints leads to
the numerical integrators mentioned above. The aim of this paper is to
propose a discrete time counterpart of the (continuous) second order
constrained Lagrangian systems. We study some basic properties of
those discrete time systems and use them to construct numerical
integrators for the continuous ones.

The plan for the paper is as follows. In Section~\ref{Sect:SOCS} we
review the notion of (continuous) second order constrained Lagrangian
system. In addition, we prove a result characterizing the evolution
with the flow of the natural Lagrangian symplectic structure of such a
system. In
Section~\ref{sec:discrete_second_order_constrained_Lagrangian_systems}
we introduce the discrete second order constrained Lagrangian systems,
their dynamics and their equations of motion. In
Sections~\ref{Sect:example1} and~\ref{Sect:example2} we apply the
discrete formalism just developed to two examples. There we find
numerical integrators and test their quality by comparing against
either the exact solution or a well known integrator of the
corresponding continuous system. On the other hand, in
Section~\ref{Sect:discflow} we prove some results about the dynamics
of the discrete systems: the existence of a well defined local flow
and a discrete analogue of the evolution of the symplectic form
studied in Section~\ref{Sect:SOCS}. Last, in
Section~\ref{Sect:future}, we comment on some directions of future
work.

\smallskip

\emph{Notation:} throughout the paper $\tau_{X}$ is the projection of
the tangent bundle $TX$ onto $X$.


\section{Second order constrained Lagrangian systems}
\label{Sect:SOCS}

In this section we review the notion of higher order constrained
system such as it appears in~\cite{ar:cendra_grillo-lagrangian_systems_with_higher_order_constraints,th:grillo-sistemas_noholonomos_generalizados}. In
particular, we shall only consider first order Lagrangian functions
(this partially excludes the systems studied in \cite{ar:krupkova-higher_order_mechanical_systems_with_constraints}). The
focus of our exposition is on systems with constraints of order at
most $2$ for the reason explained in Section~\ref{Sect:intro}. Recall
that $T^{(2)}Q$ denotes the \emph{second order tangent bundle} of the
manifold $Q$
(see~\cite{ar:crampin_sarlet_cantrijn-higher_order_differential_equations_and_higher_order_lagrangian_mechanics,bo:deleon_rodrigues-generalized_classical_mechanics_and_field_theory}).

\begin{definition} [SOCS] 
  \label{socs} 
  A \emph{second order constrained Lagrangian system} is a quadruple
  $(Q,L,C_{K},C_{V})$ where
  \begin{enumerate}
  \item $Q$ is a finite dimensional differentiable manifold, the
    \emph{configuration space},
  \item $L:TQ\rightarrow\mathbb{R}$ is a smooth function on the
    tangent bundle of $Q$, the \emph{Lagrangian},
  \item $C_{K}\subset T^{(2)}Q$ is a submanifold, the \emph{kinematic
      constraints}, and
  \item $C_{V}\subset T^{(2)}Q\times_{Q}TQ$ (where $\times_{Q}$
    denotes the fiber product on $Q$) is such that for every $q\in Q$
    and $\eta\in T_{q}^{(2)}Q$, the set
    $C_{V}|_{\eta}:=C_{V}\cap(\{\eta\}\times T_{q}Q)$, naturally
    identified with a subset of $T_{q}Q$, is either empty or a vector
    subspace, the \emph{virtual displacements} or \emph{variational
      constraints}.
  \end{enumerate}
\end{definition}

For every system of this type, the \emph{action functional} is defined
by $S(\gamma):=\int_{t_{0}}^{t_{1}}L(\gamma^{\prime}(t))\ dt$, where
$\gamma:[t_{0},t_{1}]\rightarrow Q$ is a smooth curve in $Q$ and
$\gamma^{\prime}(t)\in TQ$ is its velocity (in what follows, $\gamma
^{(2)}:[t_{0},t_{1}]\rightarrow T^{(2)}Q$ will denote its
$2$-lift). An \emph{infinitesimal variation} of $\gamma$ is a smooth
curve $\delta \gamma:[t_{0},t_{1}]\rightarrow TQ$ such that
$\tau_{Q}(\delta\gamma (t))=\gamma(t)$ $\forall t$, and it is said to
have \emph{vanishing end points} if $\delta\gamma(t_{0})=0$ and
$\delta\gamma(t_{1})=0$. The dynamics of a SOCS is determined by the
following Principle.

\begin{definition} [Lagrange--d'Alembert's Principle for SOCSs] A
  smooth curve~$\gamma :[t_{0},t_{1}]\rightarrow Q$ is a
  \emph{trajectory} of the SOCS $(Q,L,C_{K},C_{V})$ if
  \begin{enumerate}
  \item it satisfies the kinematic constraints: $\gamma^{(2)}(t)\in
    C_{K}$ $\forall t\in\lbrack t_{0},t_{1}]$; and
  \item it is a critical point of $S$ for the \emph{admissible
      variations}: $dS(\gamma)(\delta\gamma)=0$ $\forall\delta\gamma$
    with vanishing end points and such that $\delta\gamma(t)\in
    C_{V}|_{\gamma^{(2)}(t)}$ $\forall t\in\lbrack t_{0},t_{1}]$.
  \end{enumerate}
\end{definition}

\begin{remark}
  \label{rem:nonholonomic_and_socs}
  All holonomic and nonholonomic systems, \emph{i.e.}  constrained
  systems that satisfy d'Alembert's Principle, can be seen as SOCSs.
  Indeed, if we have a system $(Q,L)$ with constraints given by a
  distribution $\mathcal{D}\subset TQ$ (with $\mathcal{D}$ integrable
  in the holonomic case), defining
  $C_{K}:=(\tau^{(1,2)})^{-1}(\mathcal{D})$ and $C_{V}:=T^{(2)}%
  Q\times_{Q}\mathcal{D}$, where $\tau^{(1,2)}:T^{(2)}Q\rightarrow TQ$
  is the canonical projection, then $(Q,L,C_{K},C_{V})$ is a SOCS
  whose dynamics recovers the dynamics of the original system. With
  the same idea, generalized nonholonomic systems
  (see~\cite{ar:marle-kinematic_and_geometric_constraints_servomechanism_and_control_of_mechanical_systems,th:grillo-sistemas_noholonomos_generalizados,ar:cendra_grillo-generalized_nonholonomic_mechanics_servomechanisms_and_related_brackets}) can also be seen as SOCSs.

  Systems with (at most) second order constraints satisfying the
  natural generalization of Chetaev's Principle~\cite{ar:chetaev-on_the_gauss_principle}, as
  those appearing in \cite{ar:krupkova-higher_order_mechanical_systems_with_constraints} (with first order Lagrangians, as
  in~\cite{ar:swaczyna-mechanical_systems_with_nonholonomic_constraints_of_the_second_order}), define a particular subclass of SOCSs.

  On the other hand, second order vakonomic systems, as considered
  in~\cite{ar:benito_martindediego-hidden_simplecticity_in_hamilton_s_principle_algorithm},
  are not SOCSs because they are purely variational ---that is, their
  trajectories are critical points of the action restricted to the
  admissible paths--- and they allow Lagrangians that depend on higher
  order derivatives of the path.
\end{remark}

When $C_{V}|_{\eta}$ is nonempty for all $\eta\in C_{K}$, and $C_{V}$
is a submanifold, Theorems 17 and 19 in~\cite{ar:cendra_grillo-lagrangian_systems_with_higher_order_constraints} prove
that $\gamma$ is a trajectory of the system if and only if, $\forall
t\in [t_{0},t_{1}]$,
\begin{equation}
  \label{eq:socsmotion_kin_and_var}
  \gamma^{(2)}(t) \in C_{K} \quad\text{ and }
  \quad D_{EL}L(\gamma^{(2)}(t)) \in F_{V}|_{\gamma^{(2)}(t)},
\end{equation}
where $D_{EL}L:T^{(2)}Q\rightarrow T^{\ast}Q$ is the well known
Euler--Lagrange map (see~\cite{bo:cendra_marsden_ratiu-lagrangian_reduction_by_stages}, Thm. 2.2.3) and
$F_{V}|_{\eta}:=\left( C_{V}|_{\eta}\right) ^{\circ}$ for all $\eta\in
T^{(2)}Q$ is the \emph{space of constraint forces}. Notice that for
nonholonomic systems, given $q\in Q$ and $\eta\in T_{q}^{(2)}Q$, we
have that $F_{V}|_{\eta}= \mathcal{D}_{q}^{\circ}$ (see
Remark~\ref{rem:nonholonomic_and_socs}), that is, the constraint
forces vanish on the allowed velocities, which is the content of
d'Alembert's Principle.

Under some conditions, it is possible to define the \emph{flow}\footnote{In
this section we shall ignore issues related to  global versus local flows. For
SOCSs, there are certain conditions of  existence and uniqueness of
trajectories when $C_{V} = (\tau^{(1,2)}  \times id_{TQ})^{-1}(C_{V}^{\prime
})$ for some $C_{V}^{\prime}\subset TQ\times_{Q} TQ$  (see \cite{ar:grillo-higher_order_constrained_hamiltonian_systems},
Sect. IV).} $F_{L}:TQ\times\mathbb{R}\rightarrow TQ$ of the system. We are
interested in studying the symplecticity of the map $F^{t}_{L}:TQ\rightarrow
TQ$ corresponding to flowing for a fixed time $t$. Recall that the
\emph{Legendre transform} of $L$ is $\mathcal{F}L:TQ\rightarrow T^{\ast}Q$
defined by
\begin{equation*}
  \mathcal{F}L(v_{q})(w_{q}) := \frac{d}{dz}\bigg|_{z=0}(L(v_{q}+z\ w_{q})),
  \quad\forall v_{q},w_{q}\in T_{q}Q.  
\end{equation*}
Next, define the \emph{Lagrangian }$1$\emph{-form} $\theta_{L}\in\Omega
^{1}(TQ)$ by
\begin{equation*}
  \theta_{L}(v_{q})(V_{v_{q}}) := \mathcal{F}L(v_{q})(D\tau_{Q}(v_{q})(V_{v_{q}})),
  \quad\forall V_{v_{q}}\in T_{v_{q}}(TQ),  
\end{equation*}
and the \emph{Lagrangian }$2$\emph{-form} $\Omega_{L}\in\Omega^{2}(TQ)$ by
\begin{equation*}
  \Omega_{L} := -d\theta_{L},  
\end{equation*}
which is symplectic for regular Lagrangians. It has been shown
in~\cite{ar:cortes_martinez-non_holonomic_integrators} (Sect. 5.1) and
in~\cite{ar:deleon_martindediego_santameriamerino-geometric_integrators_and_nonholonomic_mechanics}
(Sect. II) that, for nonholonomic systems, the symplectic form
$\Omega_{L}$ is preserved by the corresponding flow $F_{L}$ up to an
additive exact form. Our next result extends this property to SOCSs
and, in particular, to generalized nonholonomic systems.

\begin{theorem} [Evolution of $\Omega_{L}$] Let $(Q,L,C_{K},C_{V})$ be
  a SOCS with flow $F_{L}:TQ\times\mathbb{R}\rightarrow TQ$ and $t$ be
  any fixed time. Then,
  \begin{equation*}
    (F_{L}^{t})^{\ast}(\Omega_{L}) = \Omega_{L} + d\nu,    
  \end{equation*}
  for $\nu\in\Omega^{1}(TQ)$ defined by
  \begin{equation*}
    \nu(q,\dot{q})(\delta q,\delta\dot{q}) := \int_{0}^{t}D_{EL}L(\gamma
    ^{(2)}(s))(\delta q(s))\ ds, \quad\ \forall(\delta q,\delta\dot{q})\in
    T_{(q,\dot{q})}(TQ),    
  \end{equation*}
  and where $\gamma$ is the trajectory with initial conditions
  $(q,\dot{q})$ and, for $s\in[0,t]$,
  \begin{equation*}
    \delta q(s):=D(\tau_{Q}\circ F_{L}^{s})(q,\dot{q})(\delta q,\delta\dot{q}) 
    \in T_{\gamma(s)}Q.    
  \end{equation*}
\end{theorem}

\begin{proof}
  The proof is based on~\cite{ar:marsden_west-discrete_mechanics_and_variational_integrators} (Sect. 1.2.3). Given a
  smooth curve $\gamma :[0,t]\rightarrow Q$ and any variation $\delta
  \gamma $ of $\gamma$ (not necessarily with vanishing end points),
  \begin{equation*}
    dS(\gamma )(\delta \gamma ) =
    \int_{0}^{t}D_{EL}L(\gamma ^{(2)}(s))(\delta \gamma (s))\ ds +
    \theta _{L}(\gamma ^{\prime }(s))(\delta \gamma (s),\ast)|_{0}^{t},
  \end{equation*}
  where $\ast $ is arbitrary but such that $(\delta \gamma (s),\ast
  )\in T_{\gamma ^{\prime }(s)}(TQ)$.  We define the restricted action
  functional $\hat{S}:TQ \rightarrow \mathbb{R}$ by
  \begin{equation*}
    \hat{S}(q,\dot{q}) := S(\hat{\gamma}),
  \end{equation*}
  where $\hat{\gamma}$ is the trajectory of the system with initial
  conditions $(q,\dot{q})$.  For all $(q,\dot{q}) \in TQ$, and all
  $(\delta q,\delta \dot{q}) \in T_{(q,\dot{q})}(TQ)$, we define a
  smooth curve in $T(TQ)$ by $(\delta q(s),\delta
  \dot{q}(s)):=D(F_{L}^{s})(q,\dot{q})(\delta q,\delta \dot{q})$,
  whose first component is an infinitesimal variation $\delta
  \hat{\gamma}$ of $\hat{\gamma}$. We compute
  \begin{eqnarray*}
    d\hat{S}(q,\dot{q})(\delta q,\delta \dot{q}) &=&dS(
    \hat{\gamma})(\delta \hat{\gamma}) \\
    &=&\int_{0}^{t}D_{EL}L(\hat{\gamma}^{(2)}(s))(\delta \hat{\gamma}(s))\
    ds+\theta _{L}(\hat{\gamma}(s))(\delta q(s),\delta \dot{q}(s))|_{0}^{t},
  \end{eqnarray*}
  where we have chosen $\ast $ to be $\delta \dot{q}(s)$
  conveniently. Rewriting the second term in the last equality as
  \begin{equation*}
    ((F_{L}^{t})^{\ast }(\theta _{L})-\theta _{L})(q,\dot{q})
    (\delta q,\delta \dot{q}),
  \end{equation*}
  we find that the $1$-form $\nu $ of the statement is $d\hat{S} -
  ((F_{L}^{t})^{\ast }(\theta _{L})-\theta _{L})$, which is well
  defined on $TQ$. Finally,
  \begin{equation*}
    d\nu =d\left( d\hat{S} - ((F_{L}^{t})^{\ast }(\theta _{L})-\theta _{L})\right)
    =d^{2}\hat{S}-((F_{L}^{t})^{\ast }(d\theta _{L})-d\theta_{L})
    =(F_{L}^{t})^{\ast }(\Omega _{L})-\Omega _{L}.
  \end{equation*}
\end{proof}

\begin{remark}
  The flow $F_{L}$ is a symplectomorphism if $d\nu=0$. When a SOCS is
  unconstrained, which, in the context of
  Remark~\ref{rem:nonholonomic_and_socs}, means that $\mathcal{D}=TQ$,
  we have that $D_{EL}L(\gamma^{(2)}(s))=0$ in the definition of
  $\nu$. Hence, in this case, $F_{L}$ is a symplectomorphism.

  In the holonomic case, \emph{i.e.} $\mathcal{D}$ is an integrable
  distribution, if $\Sigma$ is an integral submanifold of
  $\mathcal{D}$, the flow $F_{L}$ preserves $\Sigma$ and is a
  symplectomorphism with respect to the restriction of $\Omega_{L}$ to
  it. Indeed, when $(\delta q,\delta\dot{q})\in T(T\Sigma)$, we have
  that $\delta q(s)$ remains in $T\Sigma$, so that the term
  $D_{EL}L(\gamma^{(2)}(s))(\delta q(s))$ in the definition of $\nu$
  vanishes.
\end{remark}


\section{Discrete second order constrained Lagrangian systems}
\label{sec:discrete_second_order_constrained_Lagrangian_systems}

Just as SOCSs are an extension of the notion of nonholonomic system, in this
section we introduce a discrete time counterpart of SOCSs that is an extension
of the notion of discrete nonholonomic system introduced
in~\cite{ar:cortes_martinez-non_holonomic_integrators}. Later, in Section~\ref{Sect:discflow}, we study the
existence and uniqueness of trajectories and the symplectic behavior of the
discrete time evolution.

\emph{Notation:} $p_{i,j,...}^{m}$ is the projection on the $i$-th, $j$-th,
and so on, variables of $Q^{m}$ onto $Q$.

\begin{definition}[DSOCS]
  \label{DSOCS}  
  A \emph{discrete second order constrained Lagrangian system} is a
  quadruple $(Q,L_{d},D_{K},D_{V})$ where
  \begin{enumerate}
  \item $Q$ is as in Definition~\ref{socs}, 
  \item $L_{d}:Q\times Q\rightarrow\mathbb{R}$ is a smooth function,
    the \emph{discrete Lagrangian},
  \item $D_{K}\subset Q\times Q\times Q$ is a submanifold, the
    \emph{discrete kinematic constraints}, and
  \item $D_{V}\subset\left( p_{2}^{3}\right) ^{\ast}(TQ)$ (where
    $\left( p_{2}^{3}\right) ^{\ast}(TQ)$ is the pullback bundle under
    $p_{2}^{3}$) is such that for every
    $(q,q^{\prime},q^{\prime\prime})\in Q^{3}$ the subset
    $D_{V}|_{(q,q^{\prime},q^{\prime\prime})}:=D_{V}\cap(\{(q,q^{\prime}%
    ,q^{\prime\prime})\}\times T_{q^{\prime}}Q)$, naturally identified
    with a subset of $T_{q^{\prime}}Q$, is a vector subspace, the
    \emph{discrete variational constraints}.
  \end{enumerate}
\end{definition}

The \emph{discrete action functional} is defined by $S_{d}(q_{\cdot})
:= \sum_{k=0}^{N-1}L_{d}(q_{k},q_{k+1})$ where $q_{\cdot}:\{0,\ldots
,N\}\rightarrow Q$ is a discrete path in $Q$. An \emph{infinitesimal
  variation} of $q_{\cdot}$ consists of a map $\delta
q_{\cdot}:\{0,\ldots ,N\}\rightarrow TQ$ such that $\delta q_{k}\in
T_{q_{k}}Q$ $\forall k$, and it is said to have vanishing end points
if $\delta q_{0}=0$ and $\delta q_{N}=0$.  The following Principle
determines the dynamics of DSOCSs.

\begin{definition}[Discrete Lagrange--d'Alembert Principle for DSOCSs]
  \label{DSOCSdyn}  
  A discrete path $q_{\cdot}:\{0,\ldots,N\}\rightarrow Q$, with
  $N\geq2$, is a \emph{trajectory} of the DSOCS
  $(Q,L_{d},D_{K},D_{V})$ if
  \begin{enumerate}
  \item it satisfies the discrete kinematic constraints:
    \begin{equation*}
      (q_{k-1},q_{k},q_{k+1})\in D_{K}\quad\forall k\in\{1,\ldots,N-1\}, 
      \text{ and}%
    \end{equation*}
  \item it is a critical point of $S_{d}$ for the admissible
    variations: $dS_{d}(q_{\cdot})(\delta q_{\cdot})=0,$
    $\forall\delta q_{\cdot}$ with vanishing end points and such that
    \begin{equation*}
      \delta q_{k}\in D_{V}|_{(q_{k-1},q_{k},q_{k+1})}\quad\forall k\in
      \{1,\ldots,N-1\}.      
    \end{equation*}
  \end{enumerate}
\end{definition}

Let $X$ be a manifold and $X^{m}$ its $m$-th Cartesian product. When
$F:X^{m}\rightarrow\mathbb{R}^{n}$ is a smooth map, its derivative $DF$ is, in
a natural way, a differential form on $X^{m}$ with values in $\mathbb{R}^{n}$.
On the other hand, if $i_{j}:(p^{m}_{j})^{*}(TX)\rightarrow T(X^{m})$ is the
inclusion
\begin{equation*}
  i_{j}(\delta x_{j}) := (0,\ldots,0,\underbrace{\delta x_{j}}_{j}, 0,\ldots,0),
\end{equation*}
we define
\begin{equation} \label{eq:D_j}
  D_{j} F:= i_{j}^{*}(DF) = DF\circ i_{j}.
\end{equation}

When $q_{\cdot}$ is a trajectory of a DSOCS, it follows from the arbitrariness
of the admissible variations that, $\forall k\in\{1,...,N-1\}$,
\begin{equation*}
  D_{1}L_{d}(q_{k},q_{k+1})+D_{2}L_{d}(q_{k-1},q_{k}) \in\left(  
    D_{V}|_{(q_{k-1},q_{k},q_{k+1})}\right)  ^{\circ}.  
\end{equation*}
Inspired by~\cite{ar:mclachlan_perlmutter-integrators_for_nonholonomic_mechanical_systems} (Prop. 3), we define the
section $\beta$ of $D_{V}^{\ast}$ by
\begin{equation}\label{beta}
  \beta(q_{k-1},q_{k},q_{k+1}) := i_{(q_{k-1},q_{k},q_{k+1})}^{t}(\mathcal{F
  }^{+}L_{d}(q_{k-1},q_{k})-\mathcal{F}^{-}L_{d}(q_{k},q_{k+1})),
\end{equation}
where the \emph{discrete Legendre transforms} $\mathcal{F}^{-}L_{d}$
and $\mathcal{F}^{+}L_{d}:Q\times Q\rightarrow T^{\ast}Q$ are such
that $\mathcal{F}^{-}L_{d}(q,q^{\prime}) :=
(q,-D_{1}L_{d}(q,q^{\prime}))$ and $\mathcal{F}^{+}L_{d}(q,q^{\prime})
:= (q^{\prime},D_{2}L_{d}(q,q^{\prime}))$ for all $(q,q^{\prime})\in
Q\times Q$, and where $i_{\cdot} : D_{V}|_{\cdot}
\hookrightarrow\left( \left( p_{2}^{3}\right) ^{\ast}(TQ)\right)
|_{\cdot}$ is the inclusion and $i_{\cdot}^{t}$ is the transpose
map. The following result is straightforward.

\begin{theorem}
  A discrete path $q_{\cdot}:\{0,\ldots,N\}\rightarrow Q$, with
  $N\geq2$, is a trajectory of the DSOCS $(Q,L_{d},D_{K},D_{V})$ if
  and only if, $\forall k\in\{1,\ldots,N-1\}$,
  \begin{equation} \label{eq:motionkin_and_beta}
    (q_{k-1},q_{k},q_{k+1}) \in D_{K} \quad\text{ and
    }\quad\beta(q_{k-1},q_{k},q_{k+1})=0.
  \end{equation}
\end{theorem}

\begin{remark}
  \label{rem:discrete_holonomic_as_DSOCS} A discrete nonholonomic
  system as introduced in~\cite{ar:cortes_martinez-non_holonomic_integrators} is a discrete
  Lagrangian system $(Q,L_{d})$ with discrete constraint space
  $\mathcal{D}_{d}\subset Q\times Q$ (we say first order) and allowed
  variation distribution $\mathcal{D} \subset TQ$ (we say zeroth
  order). In particular, a discrete holonomic system, in the sense of
  Remark 3.3 of~\cite{ar:cortes_martinez-non_holonomic_integrators}, corresponds to the case where
  $\mathcal{D}$ is an integrable distribution and $\mathcal{D}_{d} =
  \cup_{r} \mathcal{N}_{r}\times\mathcal{N}_{r}$, where
  $\mathcal{N}_{r}$ are the integral submanifolds of $\mathcal{D}$.

  In both cases, their trajectories $q_{\cdot}=(q_{0},\ldots,q_{N})$
  are the solutions of
  \begin{equation*}
    \begin{cases}
      D_{1} L_{d}(q_{k},q_{k+1}) + D_{2} L_{d}(q_{k-1},q_{k})
      \in\mathcal{D}_{q_{k}}^{\circ},\\
      (q_{k},q_{k+1}) \in\mathcal{D}_{d}
    \end{cases}
  \end{equation*}
  for all $k=1,\ldots N-1$ and that, additionally, satisfy
  $(q_{0},q_{1})\in\mathcal{D}_{d}$. Notice that these conditions are
  equivalent to
  \begin{equation*}
    \begin{cases}
      D_{1} L_{d}(q_{k},q_{k+1}) + D_{2} L_{d}(q_{k-1},q_{k})
      \in\mathcal{D}_{q_{k}}^{\circ},\\
      (q_{k},q_{k+1}) \in\mathcal{D}_{d},\\
      (q_{k-1,}q_{k}) \in\mathcal{D}_{d}
    \end{cases}
  \end{equation*}
  for all $k=1,\ldots,N-1$. In order to ensure the existence of
  trajectories, it is usually assumed ---and we will do so--- that the
  projection $p^{2}%
  _{1}:Q\times Q\rightarrow Q$ restricted to $\mathcal{D}_{d}$ is a
  submersion (see~\cite{ar:mclachlan_perlmutter-integrators_for_nonholonomic_mechanical_systems}, Prop.~3). This last
  condition is trivially satisfied in the holonomic case.

  DSOCSs extend the discrete holonomic and nonholonomic systems as
  follows.  Given a distribution $\mathcal{D}$ and a submanifold
  $\mathcal{D}_{d}$ as above, a DSOCS $(Q,L_{d},D_{K},D_{V})$ can be
  constructed by defining
  \begin{equation*}
    D_{K} := (Q\times\mathcal{D}_{d}) \cap(\mathcal{D}_{d} \times Q) \quad\text{
      and } \quad D_{V} := (p_{2}^{3})^{*}(\mathcal{D}).    
  \end{equation*}
  Notice that $D_{K}$ is indeed a submanifold of $Q\times Q\times Q$
  because it is a transversal intersection of two submanifolds; the
  transversality condition follows from $p^{2}_{1}|_{\mathcal{D}_{d}}$
  being a submersion. It is easy to see that both systems, the
  discrete nonholonomic system and the related DSOCS, have the same
  trajectories.
\end{remark}

\begin{remark}
  Other ``higher order'' discrete mechanical systems have been
  considered in the literature. One such example is that of higher
  order discrete Lagrangian
  mechanics~\cite{ar:benito_deleon_martindediego-higher_order_discrete_lagrangian_mechanics},
  consisting of unconstrained systems with Lagrangians that may depend
  on more than two points. Also, discrete higher order vakonomic
  systems have been considered in, for
  example,~\cite{ar:colombo_martindediego_zuccalli-higher_order_discrete_variational_problems_with_constraints}. These are
  constrained systems where the Lagrangians also depend on more than
  two points and the trajectories correspond to a purely variational
  problem, just as in the continuous case mentioned in
  Remark~\ref{rem:nonholonomic_and_socs}.
\end{remark}

\begin{remark}
  From a theoretical point of view, one could be interested in a
  discrete analogue of the higher order constrained systems (in the
  sense of \cite{ar:cendra_ibort_de_leon_martin-a_generalization_of_chetaevs_principle_for_a_class_of_higher_order_nonholonomic_constraints,ar:grillo-higher_order_constrained_hamiltonian_systems}). Such an analogue can
  be obtained following ideas similar to the ones introduced in this
  section for order $2$.  For instance, a discrete kinematic
  constraint of order $k$ would be a submanifold of $Q^{k+1}$ and the
  variational constraints of order $k$ would be contained in the
  pullback bundle by $p_{j}^{k+1}:Q^{k+1}\rightarrow Q$ of $TQ$ for a
  choice of $j\in\{1,\ldots,k+1\}$.
\end{remark}


\section{Examples}
\label{sec:examples}

In this section we discuss how to apply DSOCSs to construct numerical
integrators for two (continuous) systems with second order constraints. In
each case, we picked simple discretizations to associate a discrete system to
the continuous one. Our main objective is to show how the numerical integrator
is constructed and some characteristics of its behavior. Other discretizations
and details can be found in~\cite{th:borda-sistemas_mecanicos_discretos_con_vinculos_de_orden_2}.

In this section, all angles are expressed in radians.


\subsection{Particle in the plane with prescribed signed curvature}
\label{Sect:example1}

Consider a particle in $\mathbb{R}^{2}$ forced to move with a given
signed curvature, $k:\mathbb{R}^{2}\rightarrow\mathbb{R}$, by the
effect of a force orthogonal to its velocity. For example, if the
particle is electrically charged, this could be achieved using a
magnetic field orthogonal to the plane.


\subsubsection{Continuous case}

We first describe the system in terms of Definition~\ref{socs} (see
Figure~\ref{ex1_draw} to visualize the meaning of the following
variables).

\begin{enumerate}
\item $Q:=\mathbb{R}^{2}$, with coordinates $q=(x,y)$.
\item $L((x,y),(\dot{x},\dot{y})) :=
  \frac{1}{2}m(\dot{x}^{2}+\dot{y}^{2})$, where $m$ is the mass of the
  particle.
\item Kinematic constraints: the submanifold $C_{K}\subset T^{(2)}Q$
  is defined by $\frac{d\theta}{ds}=k(x,y)$, where $\theta$ is the
  polar angle of the velocity of the particle and $ds$ is the element
  of the arc length.  Explicitly, the equation becomes
  \begin{equation}\label{exsocskin}
    \frac{\dot{x}\ \ddot{y}-\ddot{x}\ \dot{y}}{\left\Vert
        (\dot{x},\dot{y})\right\Vert ^{3}}=k(x,y).
  \end{equation}
\item Variational constraints: for each $\eta=
  ((x,y),(\dot{x},\dot{y}), (\ddot{x},\ddot{y})) \in T^{(2)}Q$, the
  subspace $C_{V}|_{\eta}$ is defined as the span of
  $(\dot{x},\dot{y})$ in $T_{(x,y)}Q$.
\end{enumerate}

In this case equation~\eqref{eq:socsmotion_kin_and_var} is equivalent
to equation~\eqref{exsocskin} together with $m\ddot{x}=\lambda\dot{y}$
and $m\ddot{y}=-\lambda\dot{x}$, where $\lambda$ is an unknown
Lagrange multiplier.

\begin{figure}[pth]
  \begin{center}
    \includegraphics[width=0.45\textwidth]{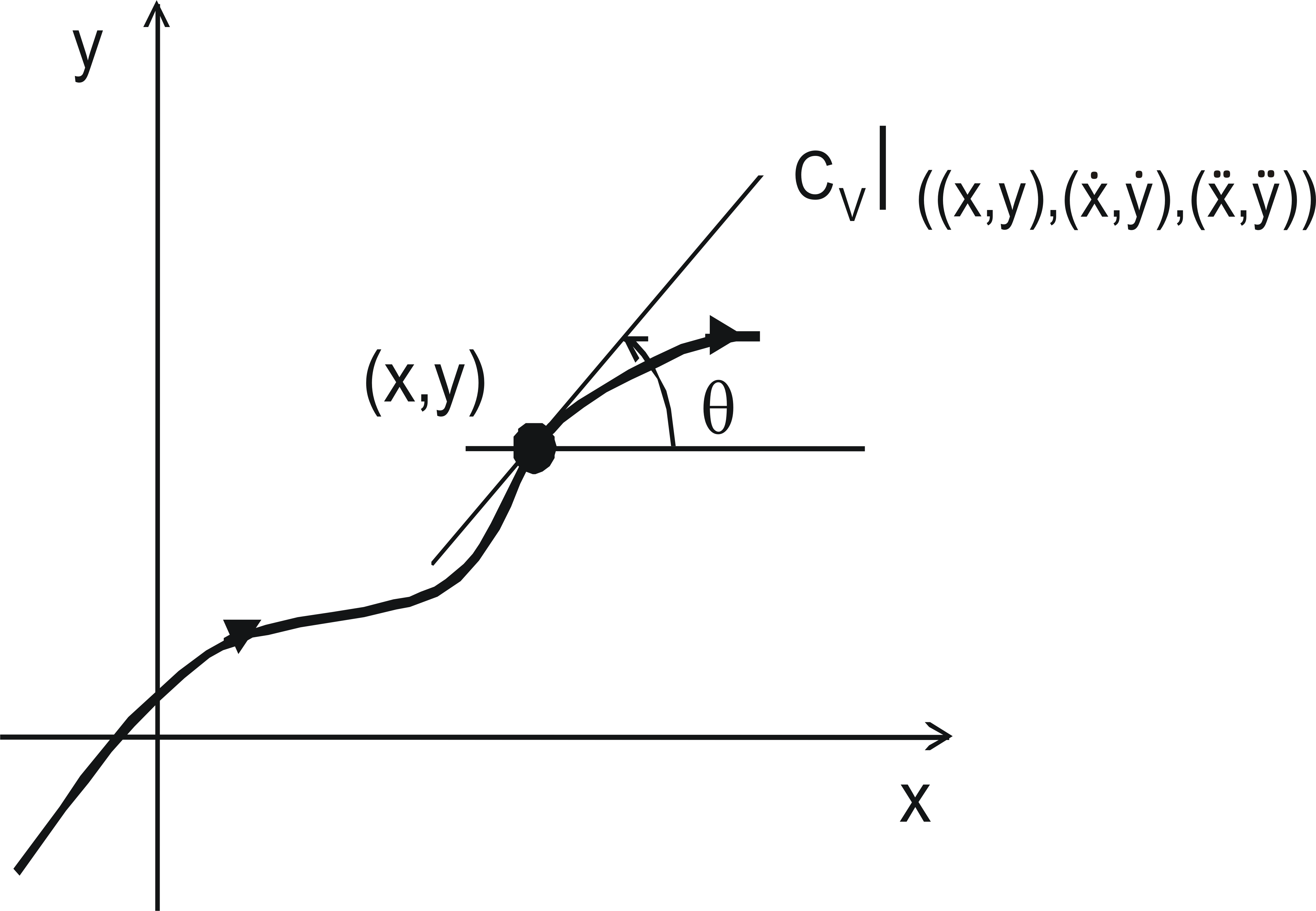}
  \end{center}
  \caption{Scheme of the particle in the plane with prescribed signed
    curvature.  The polar angle $\theta$ of the particle's velocity
    and the variational constraints at
    $((x,y),(\dot{x},\dot{y}),(\ddot{x},\ddot{y}))$ are indicated}
  \label{ex1_draw}
\end{figure}


\subsubsection{Discrete case}
\label{Sect:ex1disc}

We now associate a DSOCS to this SOCS in order to approximate its
trajectory $q(t)$ by a discrete one, $q_{\cdot}$, in such a way that
$q_{0}\approx q(0)$, $q_{1}\approx q(h)$, $q_{2}\approx q(2h)$, and so
on, where $h\in\mathbb{R}$ is the constant time step. We use the
following particular discretization process.

\begin{enumerate}
\item $Q=\mathbb{R}^{2}$.
\item $L_{d}:=L\circ\varphi_{L_{d}}^{-1}$ where $\varphi_{L_{d}} :
  TQ\rightarrow Q^{2}$ is defined in terms of its inverse by
  \begin{equation*}
    \varphi_{L_{d}}^{-1}(q_{0},q_{1}) :=\left( q_{0},\frac{q_{1}-q_{0}}{h}\right)
    .    
  \end{equation*}
\item Discrete kinematic constraints: $D_{K}:=\varphi_{D_{K}}(C_{K})$
  where $\varphi_{D_{K}}:T^{(2)}Q\rightarrow Q^{3}$ is defined by
  \begin{equation*}
    \varphi_{D_{K}}^{-1}(q_{0},q_{1},q_{2}) := \left( q_{1}, \frac{q_{2}-q_{0}%
      }{2h}, \frac{q_{2}-2q_{1}+q_{0}}{h^{2}}\right) .    
  \end{equation*}
\item Discrete variational constraints: defining $\varphi_{D_{V}}
  :=\varphi_{D_{K}}$,
  \begin{align*}
    D_{V}|_{(q_{0},q_{1},q_{2})}: &
    =C_{V}|_{\varphi_{D_{V}}^{-1}(q_{0}%
      ,q_{1},q_{2})}\\
    & =\left\langle \left\{ \left( (x_{1},y_{1}),
          \frac{x_{2}-x_{0}}{2h} \ \frac{\partial}{\partial
            x_{1}}+\frac{y_{2}-y_{0}}{2h}\ \frac{\partial }{\partial
            y_{1}} \right) \right\} \right\rangle .
  \end{align*}

\end{enumerate}

Equation~\eqref{eq:motionkin_and_beta} leads to a system of nonlinear
equations in $x_{2}$ and $y_{2}$,
\begin{align}
  \frac{\frac{\displaystyle x_{2}-x_{0}}{\displaystyle 2h}\ \frac{
      \displaystyle y_{2}-2y_{1}+y_{0}}{\displaystyle h^{2}}-\frac
    {\displaystyle x_{2}-2x_{1}+x_{0}}{\displaystyle h^{2}}\ \frac
    {\displaystyle y_{2}-y_{0}}{\displaystyle 2h}}{\left\Vert \left(
        \frac{\displaystyle x_{2}-x_{0}}{\displaystyle
          2h},\frac{\displaystyle y_{2} -y_{0} }{\displaystyle
          2h}\right) \right\Vert ^{3}} & =
  k(x_{1}, y_{1})\label{eq:ex1int1}\\
  (x_{2}-2x_{1}+x_{0})(x_{2}-x_{0}) +
  (y_{2}-2y_{1}+y_{0})(y_{2}-y_{0}) & =0.\label{eq:ex1int2}
\end{align}

To simulate the case for which $k=1$, $x(0)=y(0)=0$ and
$\dot{x}(0)=\dot {y}(0)=1$, we took different values of $h$ and solved
equations~\eqref{eq:ex1int1} and \eqref{eq:ex1int2} iteratively (using
the algorithm FindRoot of Mathematica 6.0 at each step) starting with
the discrete initial conditions $x_{0}=y_{0}=0$, $x_{1}=x_{0}+h$ and
$y_{1}=y_{0}+h$. In this situation, we know that the exact solutions
of the continuous equations of motion are
\begin{equation*}
  x(t)=\cos(\sqrt{2}t-\frac{\pi}{4})-\frac{\sqrt{2}}{2} 
  \quad\text{ and } \quad
  y(t)=\sin(\sqrt{2}t-\frac{\pi}{4})+\frac{\sqrt{2}}{2}.  
\end{equation*}

\begin{figure}[ht]
  \begin{center}
    \includegraphics[width=\textwidth]{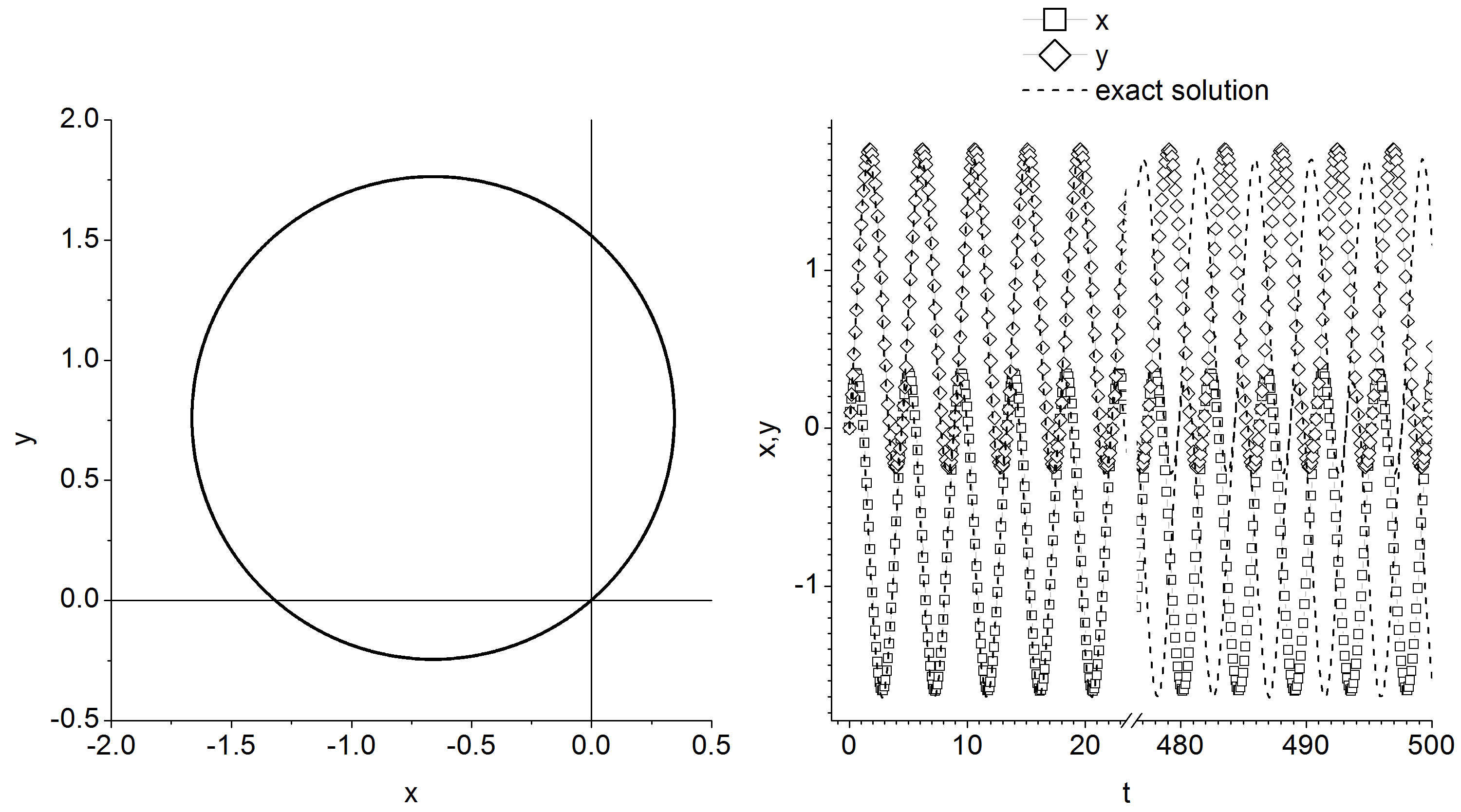}
  \end{center}
  \caption{Simulated evolution of the particle using our numerical
    integrator constructed from a DSOCS for $k=1$, $x(t)=y(t)=0$ and
    $\dot{x}(0)=\dot {y}(0)=1$. Constant time step used:
    $h=0.1$. LEFT: trajectory on the plane, RIGHT: comparison between
    our approximation and the exact solutions of $x$ and $y$ over two
    time intervals}
  \label{ex1_trajectory_and_xy}
\end{figure}

On the one hand, we found our results satisfactory at a qualitative
level (see Figure~\ref{ex1_trajectory_and_xy} corresponding to
$h=0.1$): as expected, the trajectory in the plane is a circumference
of radius $1$ which passes through the origin and is tangent to the
line of slope $1$ at that point; there are no changes in the amplitude
and the frequency of the oscillations of $x$ and $y$ during the time
of simulation $[0,500]$. This good behavior may be partially due to
the following property of the system: since each summand in
equation~\eqref{eq:ex1int2} is a difference of squares,
$(x_{2}-x_{1})^{2}-(x_{1}-x_{0})^{2}+(y_{2}-y_{1})^{2}-(y_{1}-y_{0})^{2}$
equals zero, so we have that $L_{d}(q_{0},q_{1})=L_{d}(q_{1},q_{2})$,
\emph{i.e.} our numerical integrator preserves the (discretized)
energy of the system as it occurs in the continuous case. Apart from
that, we can also say our integrator is symmetric
\cite{bo:hairer_lubich_wanner-geometric_numerical_integration}.

On the other hand, on the right side of Figure~\ref{ex1_trajectory_and_xy} we
see how the simulated evolution is slowly left behind by the exact solution.
Their maximum difference occurs near $t=500$. This maximum difference over the
$[0,500]$ time interval is what we take for the error of the numerical
integrator. Figure~\ref{ex1_errors} shows the error for several values of the
time step $h$. The slope of the line shown in the graph ($\approx1.6$)
suggests that the integrator is convergent of order $1$, according to Section
2.2.2 of~\cite{ar:marsden_west-discrete_mechanics_and_variational_integrators}.

\begin{figure}[h]
  \begin{center}
    \includegraphics[width=0.5\textwidth]{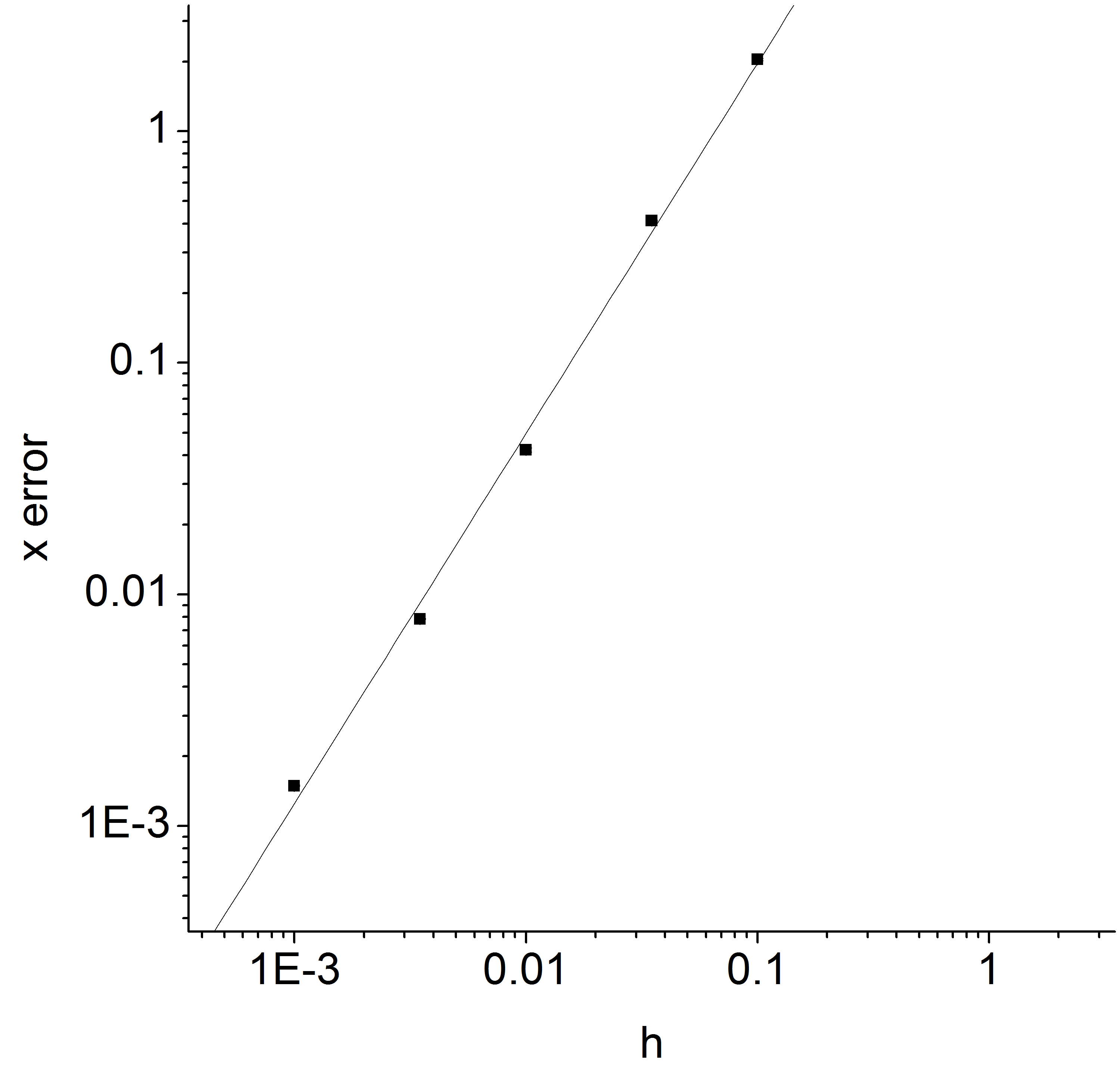}
  \end{center}
  \caption{Plot of the $x$-coordinate error vs $h$, using logarithmic
    scales}
  \label{ex1_errors}
\end{figure}


\subsection{Inertia wheel pendulum with a Lyapunov constraint}
\label{Sect:example2}

In Reference~\cite{ar:grillo_marsden_nair-lyapunov_constraints_and_global_asymptotic_stabilization}, a method for asymptotic
stabilization of underactuated mechanical systems has been studied. It
consists of: $(1)$ impose on the system a second order constraint of
the form
\begin{equation}\label{eq:lyapconst}
  \frac{dV}{dt}(q(t),\dot{q}(t))=-F(q(t),\dot{q}(t)),
\end{equation}
the so-called \emph{Lyapunov constraints}, where
$F,V:TQ\rightarrow\mathbb{R}$ are nonnegative functions with $V$
proper and vanishing only at the desired equilibrium point; and $(2)$
find the related constraint force (to be implemented by the
actuators), which would play the role of the control law. It is clear
that, if the system satisfies the previous constraints, then
$V(q(t),\dot{q}(t))$ decreases over time, resulting in a Lyapunov
function. In order to ensure the existence of a related constraint
force, $V$ must satisfy a PDE that depends on the actuators.

\begin{figure}[b]
  \begin{center}
    \includegraphics[width=0.375\textwidth]{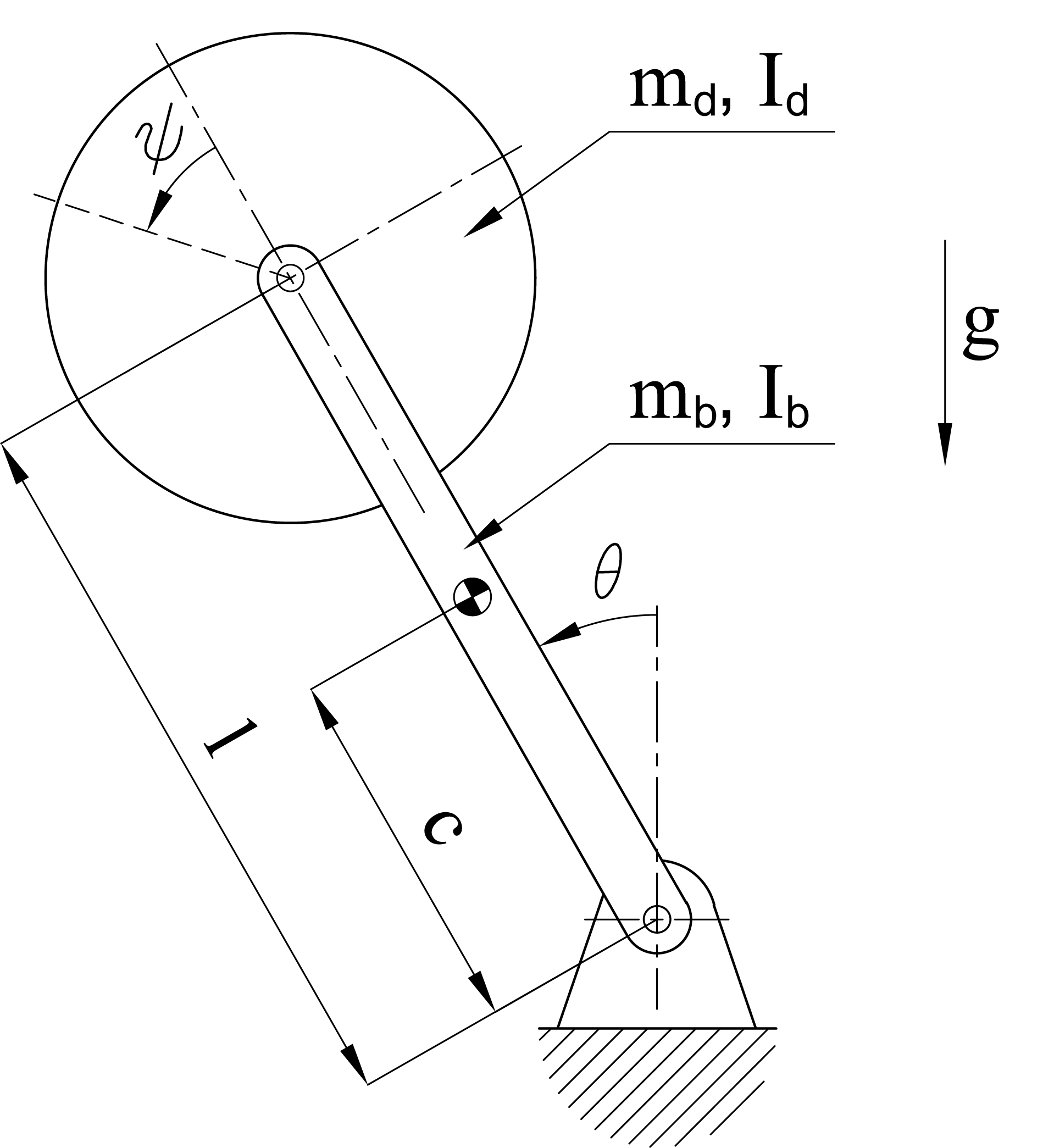}
  \end{center}
  \caption{Scheme of the inertia wheel pendulum. Some of the physical
    parameters associated to its components (masses, lengths and
    moments of inertia) as well as the coordinates used are
    indicated}
  \label{ex2_draw}
\end{figure}

It can be shown, in general, that the underactuated system
(\emph{i.e.} the mechanical system and the actuators) together with
the Lyapunov constraint define a SOCS. In the case of the inertia
wheel pendulum with one actuator on the wheel (see
Figure~\ref{ex2_draw}), if we want to asymptotically stabilize it at
the upright position, we can use the functions $V$ and $F$ found
in~\cite{ar:grillo_marsden_nair-lyapunov_constraints_and_global_asymptotic_stabilization} (Sect.~5.1). The SOCS defined by the
inertia wheel pendulum and the mentioned Lyapunov constraint is
described below.


\subsubsection{Continuous case}

We start by adapting the Hamiltonian description of the system given
in~\cite{ar:grillo_marsden_nair-lyapunov_constraints_and_global_asymptotic_stabilization} to the variational formulation of SOCSs.

\begin{enumerate}
\item $Q:=S^{1}\times S^{1}$, with coordinates $q=(\theta,\psi)$.
\item $L(\theta,\psi,\dot{\theta},\dot{\psi}):=
  \frac{1}{2}I\dot{\theta}^{2}+\frac{1}{2}J(\dot{\theta}+\dot{\psi})^{2}
  -\tilde{M}g(1+\cos(\theta))$, where $g$ is the acceleration of
  gravity and $I$, $J$ and $\tilde{M}$ are defined in terms of the
  masses, moments of inertia and characteristic lengths of the
  components of the system by $I:=m_{b}c^{2}+m_{d}l^{2}+I_{b}$,
  $J:=I_{d}$ and $\tilde{M}:=m_{b}c+m_{d}l$.
\item Kinematic constraints: the submanifold $C_{K} \subset T^{(2)}Q$
  is defined by equation~\eqref{eq:lyapconst} by choosing\footnote{The
    choices, explained in detail in~\cite{ar:grillo_marsden_nair-lyapunov_constraints_and_global_asymptotic_stabilization}, are
    aimed at making $V$ an energy-like function and $F$ a bounded
    function satisfying certain relations to guarantee the realization
    of the system as an actuated system under a bounded control
    signal.}
  \begin{equation*}
    \begin{split}
      V(\theta,\psi,\dot{\theta},\dot{\psi}):= & \
      \frac{1}{2}f[(I+J)\dot{\theta
      }+J\dot{\psi}]^{2}\\
      & +\frac{1}{2}h_{c}J^{2}(\dot{\theta}+\dot{\psi})^{2}+
      g_{c}J[(I+J)\dot
      {\theta}+J\dot{\psi}](\dot{\theta}+\dot{\psi})\\
      &  +\chi\lbrack1-\cos(\psi-n\theta)]+\frac{Me}{d}(1-\cos(\theta)),\\
      F(\theta,\psi,\dot{\theta},\dot{\psi}):= & \ 
      \rho\ \tanh\{g_{c}[(I+J)\dot{\theta}+J\dot{\psi}]+ 
      h_{c}J(\dot{\theta}+\dot{\psi})\}\ \\
      & \cdot\{g_{c}[(I+J)\dot{\theta}+J\dot{\psi}]+
      h_{c}J(\dot{\theta}+\dot{\psi })\},
    \end{split}
  \end{equation*}
  where $\chi,\rho, d, e >0$, $M:=\tilde{M}g$,
  $h_{c}:=d\frac{nb-c}{ac-b^{2}}$, $g_{c}:=d\frac{na-b}{ac-b^{2}}$,
  $f:=\frac{g_{c}^{2}+e}{h_{c}}$ with $a:=\frac{1}{I}$,
  $b:=-\frac{1}{I}$, $c:=\frac{1}{I}+\frac{1}{J}$, and
  $n\in\mathbb{Z}$ is such that $nb>c$. Note
  that~\eqref{eq:lyapconst} becomes a second order differential
  equation.
\item Variational constraints: for $\eta=
  ((\theta,\psi),(\dot{\theta},
  \dot{\psi}),(\ddot{\theta},\ddot{\psi})) \in T^{(2)}Q$, the subspace
  $C_{V}|_{\eta}$ is defined as the span of
  $\frac{\partial}{\partial\theta }|_{(\theta,\psi)}$ in
  $T_{(\theta,\psi)}Q$.
\end{enumerate}

Then, the trajectory conditions~\eqref{eq:socsmotion_kin_and_var}
become~\eqref{eq:lyapconst} and the system $\tilde{M}g\ \sin(\theta
)-I\ddot{\theta}-J(\ddot{\theta}+\ddot{\psi})=0$,
$-J(\ddot{\theta}+\ddot {\psi})=\lambda$, where $\lambda$ is an
unknown lagrange multiplier.


\subsubsection{Discrete case}

We want to construct a numerical integrator of the equations of motion
of this SOCS to provide an approximation of $q(t)$ as in
Section~\ref{Sect:ex1disc}.  From now on, we replace $S^{1}\times
S^{1}$ with its universal covering space $\mathbb{R}^{2}$ and adapt
all the elements of our SOCS to this new configuration space, which
can be done easily by letting $(\theta,\psi)$ vary over all the
plane. Physically, we capture the same dynamics by doing so but, for
practical issues, this allows us to discretize the whole $TQ$ space by
using a diffeomorphism onto $Q\times Q$. Recalling the discretizations
$\varphi_{L_{d}}$ and $\varphi_{D_{K}}$ used in
Section~\ref{Sect:ex1disc}, we propose the following DSOCS:

\begin{enumerate}
\item $Q:=\mathbb{R}^{2}$.
\item $L_{d}:=L\circ\varphi_{L_{d}}^{-1}$.
\item Discrete kinematic constraints: $D_{K}:=\varphi_{D_{K}}(C_{K})$.
\item Discrete variational constraints: $D_{V}|_{(q_{0},q_{1},q_{2})}
  := C_{V}|_{\varphi_{D_{V}}^{-1}(q_{0},q_{1},q_{2})}$.
\end{enumerate}

The second condition of~\eqref{eq:motionkin_and_beta} leads to
\begin{equation}\label{ex2eq1}
  (I+J)\frac{(\theta_{2}-2\theta_{1}+\theta_{0})}{h^{2}} + 
  J\frac{(\psi_{2}-2\psi_{1}+\psi_{0})}{h^{2}}-M\sin(\theta_{1})=0.
\end{equation}
Substituting $\psi_{2}$ from~\eqref{ex2eq1} into the first condition
of~\eqref{eq:motionkin_and_beta} leads to a nonlinear equation
involving only $\theta_{2}$,
\begin{equation}\label{ex2eq2}
  A\theta_{2}^{2}+B\theta_{2}+C=-\rho\tanh(D\theta_{2}+E)
  (D\theta_{2}+E),
\end{equation}
where $A$ equals the constant $-\frac{dI^{2}(I+(n+1)J)}{2h^{3}}$, and
$B$, $C$, $D$ and $E$ depend on the system constants, the time step
$h$, and the initial data $\theta_{0}$, $\theta_{1}$, $\psi_{0}$ and
$\psi_{1}$.

We used this DSOCS as a numerical integrator and tested it with
parameters $I=312.5$, $J=2.0772$, $M=37.98$, $d=1$, $e=1000$,
$\chi=100$, $n=-154$, $\rho=2$, and initial conditions
$\theta(0)=0.5$, $\psi(0)=0$, $\dot{\theta }(0)=0$,
$\dot{\psi}(0)=0.5$. We took different values of $h$ and
solved~\eqref{ex2eq2} iteratively (using the algorithm FindRoot of
Mathematica 6.0 and then calculating $\psi_{2}$ using~\eqref{ex2eq1})
starting with the discrete initial conditions $\theta_{0}=0.5$,
$\psi_{0}=0$, $\theta_{1} = \theta_{0}$ and $\psi_{1}=\psi_{0}+0.5h$.

\begin{figure}[b]
  \begin{center}
    \includegraphics[width=1.\textwidth]{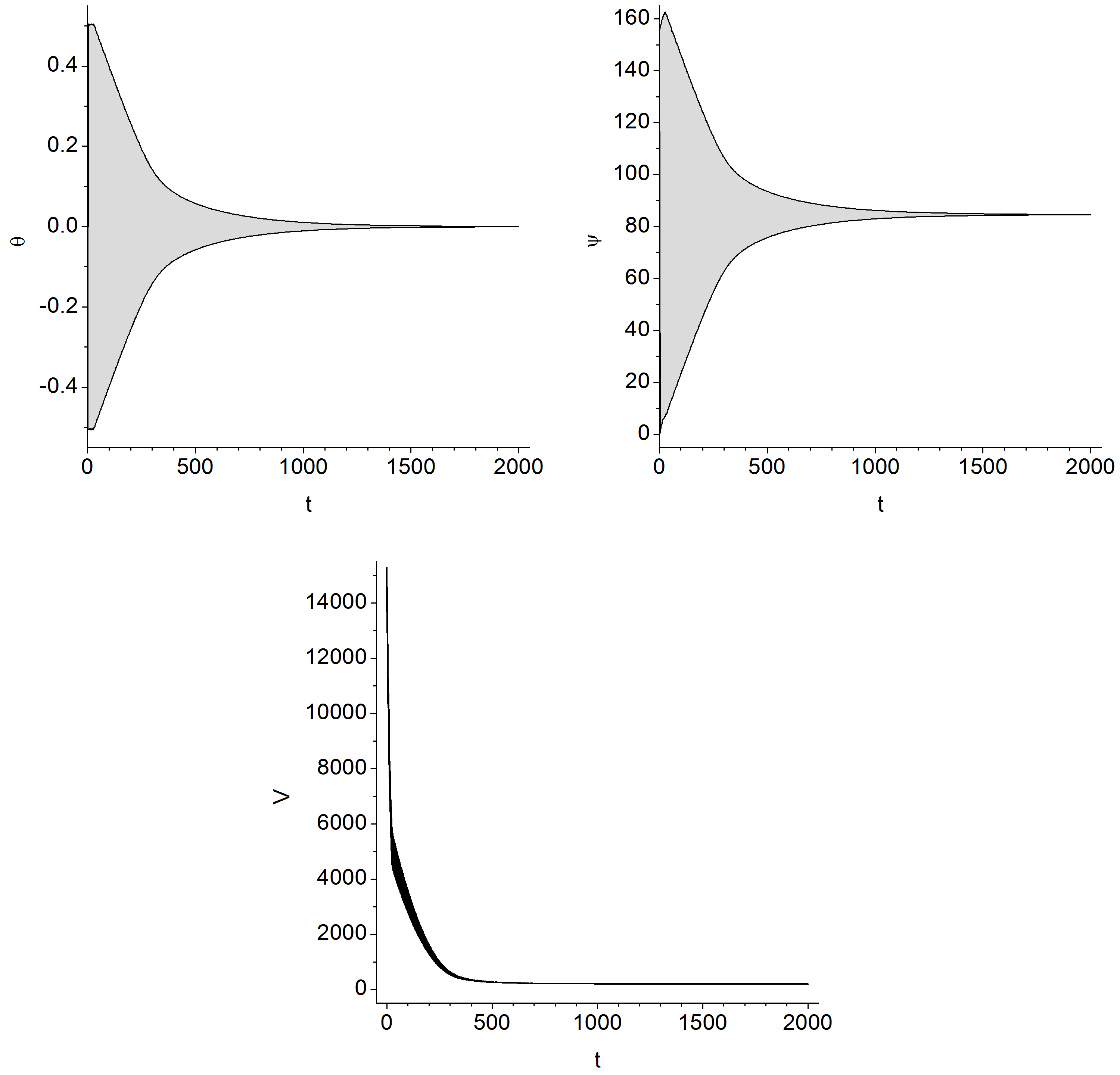}
  \end{center}
  \caption{Simulated evolution of $\theta$, $\psi$ and $V$ using our
    numerical integrator constructed from a DSOCS for the initial
    conditions $\theta (0)=0.5$, $\psi(0)=0$, $\dot{\theta}(0)=0$,
    $\dot{\psi}(0)=0.5$. Constant time step used: $h=0.1$. The gray
    area in the first two graphs corresponds to a fast oscillation}
  \label{ex2_th_psi_V}
\end{figure}

Solutions of equation~\eqref{ex2eq2}, whenever they exist, usually
come in pairs, but in order to simulate the evolution of our SOCS we
had to choose one. This phenomenon is a consequence of the equations
of motion being algebraic equations ---rather than differential
equations--- and, so, it is present in all types of discrete
mechanical systems, including DSOCSs. For the present example, we
adopted the criterion of picking the solution that is closer to the
previous position $\theta$ in each step. However, this works as long
as the two candidate solutions are sufficiently apart. When this does
not occur, we noticed that the correct behavior is obtained by
choosing the solution that decreases $F$ and, consequently $V$, as
desired.

To test the behavior of our numerical integrator, we used the output
of the sophisticated algorithm NDSolve of Mathematica 6.0 as the exact
solution.  Figure~\ref{ex2_th_psi_V} corresponds to a time step
$h=0.1$; the plots obtained with NDSolve are omitted in there because
they are indistinguishable from those coming from our simulations, at
least, for the scales used in the figure. Hence, our simulations are
consistent qualitatively with the one provided by NDSolve. The
coordinates $\theta$ and $\psi$ exhibit damped oscillatory behavior in
time associated to the asymptotic stabilization of the pendulum at its
upright position ($t \approx1000$); as it is required by the kinematic
constraint, the value of the Lyapunov function decreases with time
tending to zero ($t \approx500$).

As in the previous example, we use the maximum difference between the
numerical integrator and NDSolve solutions over the $[0,2000]$ time
interval as the error of the numerical
integrator. Figure~\ref{ex2_errors} shows the error for several values
of the time step $h$. The slope of the line shown in the graph
($\approx1.3$) suggests that the integrator is convergent of order
$1$, according to Section 2.2.2 of~\cite{ar:marsden_west-discrete_mechanics_and_variational_integrators}.

\begin{figure}[h]
  \begin{center}
    \includegraphics[width=0.52\textwidth]{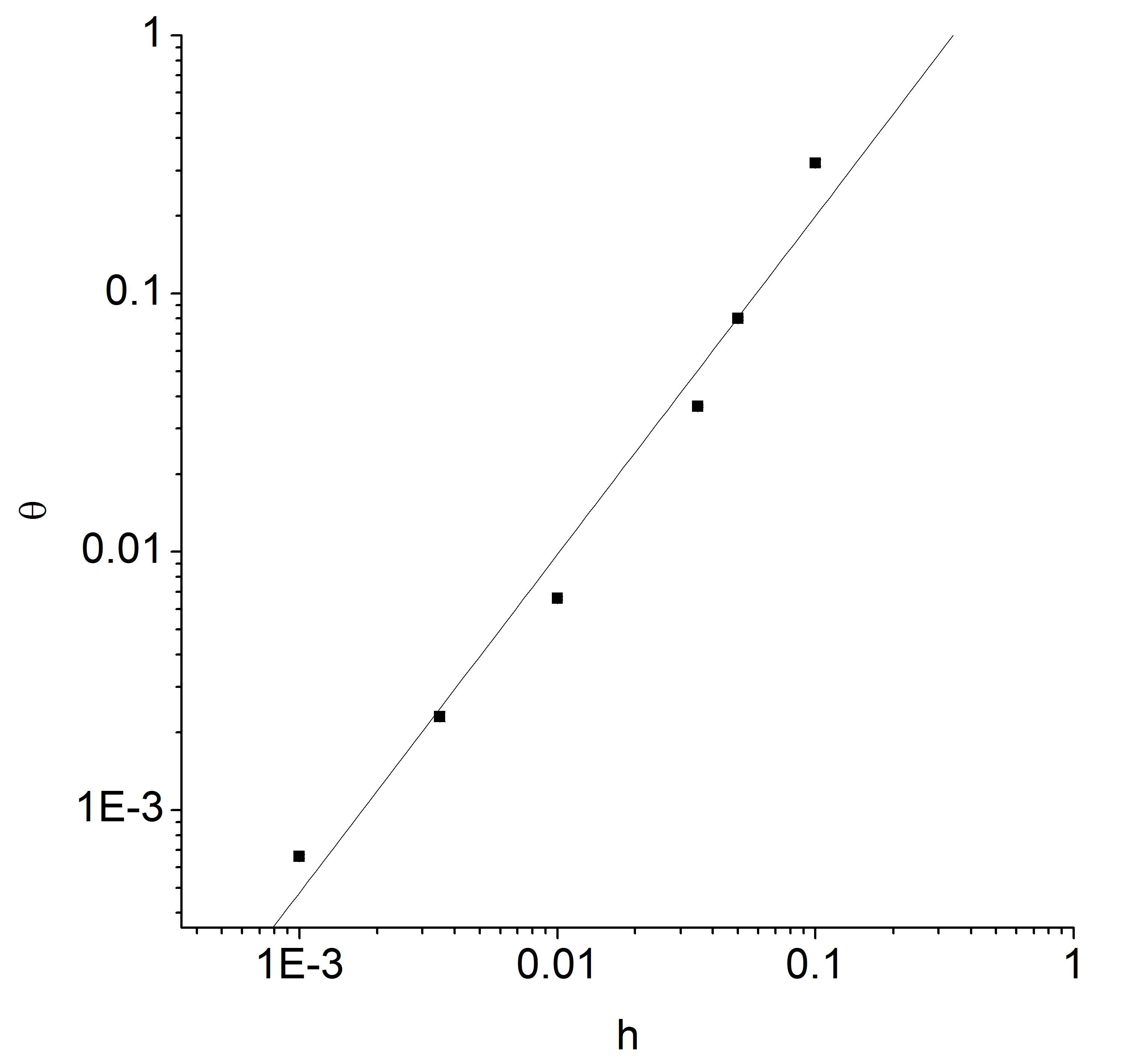}
  \end{center}
  \caption{Plot of the $\theta$-coordinate error vs $h$, using
    logarithmic scales}
  \label{ex2_errors}
\end{figure}


\section{Some properties of the discrete flow}
\label{Sect:discflow}

In this section we study the evolution of a DSOCS from the point of
view of a discrete flow function. Let $(Q,L_{d},D_{K},D_{V})$ be a
DSOCS such that $D_{V}$ is a vector subbundle of
$(p^{3}_{2})^{*}(TQ)$. Fix a trajectory $(q_{0},q_{1},q_{2})$ of the
system and an open set $U\subset Q\times Q\times Q$ containing it. It
is convenient to choose a smooth map $\phi:D_{V}^{\ast
}|_{U}\rightarrow\mathbb{R}^{n_{V}}$ such that $\phi^{-1}(\{0\})$ is
the image of the zero section of $D_{V}^{\ast}|_{U}$ (locally, this
imposes no restriction). Then, we have the following existence result.

\begin{theorem} [Discrete flow] 
  \label{discreteflow} 
  Assume that the DSOCS described above also satisfies the following
  conditions.
  \begin{enumerate}
  \item \label{it:constant_rank} $\phi\circ\beta|_{D_{K}\cap U}$ has
    constant rank,
  \item \label{it:injectivity} The restrictions of
    $D_{3}(\phi\circ\beta|_{U})(q_{0},q_{1},q_{2})$ and
    $D_{1}(\phi\circ\beta|_{U})(q_{0},q_{1},q_{2})$ to the subspace
    $T_{(q_{0},q_{1},q_{2})}D_{K}$ are injective (see~\eqref{eq:D_j}).
  \end{enumerate}
  Then, there exists a diffeomorphism $F_{L_{d}}:C_{d}\rightarrow
  F_{L_{d}}(C_{d})$, called \emph{discrete flow}, between submanifolds of
  $Q\times Q$ containing $(q_{0},q_{1})$ and $(q_{1},q_{2})$,
  respectively, such that
  \begin{enumerate}[i.] 
  \item \label{it:flows-FLd_maps_trajectory}
    $F_{L_{d}}(q_{0},q_{1})=(q_{1},q_{2})$ and
  \item \label{it:flows-FLd_is_a_flow}
    $(\hat{q}_{0},\hat{q}_{1},(p_{2}^{2}\circ
    F_{L_{d}})(\hat{q}_{0},\hat{q}_{1}))$ is a trajectory $\forall
    (\hat{q}_{0},\hat{q}_{1})\in C_{d}$.
  \end{enumerate}
\end{theorem}

\begin{proof}
  Section $\beta $ defined in~\eqref{beta} is smooth due to the
  smoothness of $D_{V}$. From condition~\ref{it:constant_rank} in the
  statement, $W:=\left( \phi \circ \beta |_{D_{K}\cap U}\right)
  ^{-1}(\{0\})$ is a submanifold of $D_{K}\cap U$. All the elements of
  $W$ are trajectories since they are the triples which satisfy
  condition~\eqref{eq:motionkin_and_beta}.  On the other hand, as $
  \ker (Dp_{1,2}^{3}|_{W}(q_{0},q_{1},q_{2}))= \ker (D(\phi \circ
  \beta |_{D_{K}\cap U})(q_{0},q_{1},q_{2}))\cap
  T_{(q_{0},q_{1},q_{2})}(\{q_{0}\}\times \{q_{1}\}\times Q)$, which
  vanishes by condition~\ref{it:injectivity} in the statement,
  $p_{1,2}^{3}|_{W}$ is a local immersion at $(q_{0},q_{1},q_{2})$.
  It follows that $p_{1,2}^{3}|_{W}$ is a local diffeomorphism between
  a neighborhood $B\subset W$ of $(q_{0},q_{1}\,,q_{2})$ and a
  submanifold of $Q\times Q$ containing $ (q_{0},q_{1})$. Analogously,
  by condition~\ref{it:injectivity} in the statement,
  $p_{2,3}^{3}|_{W}$ is a local diffeomorphism between a neighborhood
  $B^{\prime }\subset W$ of $(q_{0},q_{1},q_{2})$ and a submanifold of
  $ Q\times Q$ containing $(q_{1},q_{2})$.  Finally, let
  $C_{d}:=p_{1,2}^{3}(B\cap B^{\prime })$ and define
  \begin{equation*}
    F_{L_{d}}:C_{d}\rightarrow Q\times Q \quad \text{ by } \quad F_{L_{d}} :=
    p_{2,3}^{3}\circ (p_{1,2}^{3}|_{B\cap B^{\prime}})^{-1}.
  \end{equation*}
  Then $C_d$ and $F_{L_{d}}(C_{d}) = p_{2,3}^{3}(B\cap B^{\prime})$
  are submanifolds of $Q\times Q$, $F_{L_{d}}:C_{d}\rightarrow
  F_{L_{d}}(C_{d})$ is a diffeomorphism and
  conditions~\ref{it:flows-FLd_maps_trajectory}
  and~\ref{it:flows-FLd_is_a_flow} in the statement are satisfied.
\end{proof}

\begin{remark}
  \label{rem:discrete_holonomic_and_flows} When a DSOCS comes from a
  discrete holonomic system (see
  Remark~\ref{rem:discrete_holonomic_as_DSOCS}), we have that
  \begin{equation*}
    C_{d} \subset p^{3}_{1,2}(D_{K}) \subset\mathcal{D}_{d} = \cup_{r}
    \mathcal{N}_{r}\times\mathcal{N}_{r}.    
  \end{equation*}
  Let $C_{d,r}:= C_{d} \cap(\mathcal{N}_{r}\times\mathcal{N}_{r})$. It
  is easy to check that $F_{L_{d}}(C_{d,r}) =
  F_{L_{d}}(C_{d})\cap(\mathcal{N}_{r}\times\mathcal{N}_{r})$.
\end{remark}


Let $Q$ and $L_{d}$ be as in Definition~\ref{DSOCS}. Following the literature
(see~\cite{ar:cortes_martinez-non_holonomic_integrators}), we define the \emph{discrete Lagrangian
$1$-forms} $\theta_{L_{d}}^{-}$, $\theta_{L_{d}}^{+}\in\Omega^{1}(Q\times Q)$
by
\begin{equation}\label{theta+-}
  \begin{split}
    \theta_{L_{d}}^{-}(q,q^{\prime})(v_{q},v_{q^{\prime}}) : &
    =\mathcal{F}^{-}L_{d}(q,q^{\prime})(v_{q})\\
    \theta_{L_{d}}^{+}(q,q^{\prime})(v_{q},v_{q^{\prime}}) : &
    =\mathcal{F}^{+}L_{d}(q,q^{\prime})(v_{q^{\prime}})
  \end{split}
\end{equation}
for all $(v_{q},v_{q^{\prime}})\in T_{(q,q^{\prime})}(Q\times Q)$. In
addition, we define the \emph{discrete Lagrangian $2$-form}
$\Omega_{L_{d}} \in\Omega^{2}(Q\times Q)$ as $\Omega_{L_{d}}:=
-d\theta_{L_{d}}^{+} = -d\theta_{L_{d}}^{-}$ (the last equality is
true because $dL_{d}=\theta _{L_{d}}^{+}-\theta_{L_{d}}^{-}$). It can
be seen that, under certain conditions of regularity on $L_{d}$,
$\Omega_{L_{d}}$ is a symplectic form.

\begin{theorem}[Evolution of $\Omega_{L_{d}}$] 
  Let $(Q,L_{d},D_{K},D_{V})$ be a DSOCS with discrete flow $
  F_{L_{d}}:C_{d}\rightarrow F_{L_{d}}(C_{d})$. Also, let
  $\Omega_{L_{d}}^{C_{d}}\in\Omega^{2}(C_{d})$ and
  $\Omega_{L_{d}}^{F_{L_{d}}(C_{d})}\in\Omega^{2}(F_{L_{d}}(C_{d}))$
  be the restrictions of $\Omega_{L_{d}}$ to the corresponding
  submanifolds of $Q\times Q$. Then,
  \begin{equation*}
    (F_{L_{d}})^{\ast}\left( \Omega_{L_{d}}^{F_{L_{d}}(C_{d})}\right)  =
    \Omega_{L_{d}}^{C_{d}}+d\xi,    
  \end{equation*}
  where $\xi\in\Omega^{1}(C_{d})$ is defined by
  \begin{equation}\label{eq:xidef}
    \xi(q_{0},q_{1})(\delta q_{0},\delta q_{1}) := 
    (\mathcal{F}^{+}L_{d}(q_{0},q_{1}) - 
    \mathcal{F}^{-}L_{d}(F_{L_{d}}(q_{0},q_{1})))(\delta q _{1})
  \end{equation}
  for all $(q_{0},q_{1})\in C_{d}$, and all $(\delta q_{0}, \delta
  q_{1})\in T_{(q_{0},q_{1})}C_{d}$.
\end{theorem}

\begin{proof}
  The proof is based on \cite{ar:marsden_west-discrete_mechanics_and_variational_integrators} (Sect. 1.3.2). Let
  $(q_{0},q_{1})\in C_{d}$ and $(\delta q_{0},\delta q_{1})\in
  T_{(q_{0},q_{1})}C_{d}$. If $q_2 := (p_{2}^{2}\circ
  F_{L_{d}})(q_{0},q_{1})$ we can interpret $(\delta q_{0},\delta
  q_{1})$ as an infinitesimal variation of the initial condition
  inducing the infinitesimal variation $\delta q_2 := D(p_{2}^{2}\circ
  F_{L_{d}})(q_{0},q_{1})(\delta q_{0},\delta q_{1})$ over $q_2$.
  Define the restricted discrete action functional $\hat{S}_{d} :
  C_{d}\rightarrow \mathbb{R}$ by
  \begin{equation*}
    \hat{S}_{d}(q_{0},q_{1}) := S_{d}(q_{0}, q_{1},
    (p_{2}^{2}\circ F_{L_{d}})(q_{0},q_{1})).
  \end{equation*}
  From the definitions of the Lagrangian $1$-forms~\eqref{theta+-} it
  is easy to see that
  \begin{eqnarray*}
    d\hat{S}_{d}(q_{0},q_{1})(\delta q_{0},\delta q_{1})
    &=&dS_{d}(q_{0},q_{1},q_{2})(\delta q_{0},\delta
    q_{1},D(p_{2}^{2}\circ F_{L_{d}})(q_{0},q_{1})(\delta q_{0},\delta q_{1})) \\
    &=&(\mathcal{F}^{+}L_{d}(q_{0},q_{1})-
    \mathcal{F}^{-}L_{d}(q_{1},
    (p_{2}^{2}\circ F_{L_{d}})(q_{0},q_{1})))(\delta q_{1}) \\
    &&+\left[ \theta _{L_{d}}^{+}(q_{1},q_{2})(\delta q
      _{1},\delta q_{2})-\theta _{L_{d}}^{-}(q_{0},q_{1})(\delta
      q_{0},\delta q_{1})\right] .
  \end{eqnarray*}
  The bracketed term in the last sum is
  \begin{equation*}
    (F_{L_{d}}^{\ast }((\theta _{L_{d}}^+)^{F_{L_{d}}(C_{d})})-(\theta
    _{L_{d}}^-)^{C_{d}})(q_{0},q_{1})(\delta q_{0},\delta q_{1}),
  \end{equation*}
  where $(\theta _{L_{d}}^-)^{C_{d}}$ and $(\theta
  _{L_{d}}^+)^{F_{L_{d}}(C_{d})}$ are the restrictions of
  $\theta_{L_d}^-$ and $\theta_{L_d}^+$ to $\Omega^2(C_d)$ and
  $\Omega^2(F_{L_d}(C_d))$, respectively. Since $(q_{0},q_{1})$ and
  $(\delta q_{0},\delta q_{1})$ are arbitrary, using~\eqref{eq:xidef}
  we obtain
  \begin{equation*}
    \xi =d\hat{S}_{d} -(F_{L_{d}}^{\ast }((\theta _{L_{d}}^+)^{F_{L_{d}}(C_{d})}) -
    (\theta_{L_{d}}^-)^{C_{d}}).
  \end{equation*}
  Therefore,
  \begin{equation*}
    \begin{split}
      d\xi =&\ d(d\hat{S}_{d} -(F_{L_{d}}^{\ast }((\theta
      _{L_{d}}^+)^{F_{L_{d}}(C_{d})}) -
      (\theta_{L_{d}}^-)^{C_{d}})) \\
      =&\ d^2\hat{S}_{d} + F_{L_{d}}^{\ast
      }(-d(\theta_{L_{d}}^+)^{F_{L_{d}}(C_{d})}) -
      (-d(\theta_{L_{d}}^-)^{C_{d}}) \\
      =&\ F_{L_{d}}^{\ast }(\Omega _{L_{d}}^{F_{L_{d}}(C_{d})}) -
      \Omega_{L_{d}}^{C_{d}}.
    \end{split}
  \end{equation*}
\end{proof}

\begin{remark}
  The flow $F_{L_{d}}$ is a symplectomorphism if $d\xi=0$. It follows
  from~\eqref{eq:xidef} that $\xi$ vanishes when $\delta q_{1} \in
  D_{V}|_{(q_{0},q_{1},F_{L_{d}}(q_{0},q_{1}))}$. This situation
  occurs, for instance, when a DSOCS comes from an unconstrained
  system, where $D_{V} |_{(q_{0},q_{1},F_{L_{d}}(q_{0},q_{1}))} =
  T_{q_{1}}Q$. It also occurs when it comes from a discrete holonomic
  system (see Remark~\ref{rem:discrete_holonomic_as_DSOCS}). Indeed if
  $(q_{0},q_{1}) \in C_{d,r}$ (see
  Remark~\ref{rem:discrete_holonomic_and_flows}) and $(\delta q_{0},
  \delta q_{1}) \in T_{(q_{0},q_{1})}C_{d,r}$, we have that $\delta
  q_{1} \in T_{q_{1}} \mathcal{N}_{r} =
  D_{V}|_{(q_{0},q_{1},F_{L_{d}}(q_{0},q_{1}))}$, so that
  $\xi(q_{0},q_{1})(\delta q_{0},\delta q_{1}) = 0$.  Hence, under
  these conditions,
  \begin{equation*}
    (F_{L_{d}}|_{C_{d,r}})^{*}\left( \Omega_{L_{d}}^{F_{L_{d}}(C_{d,r})}\right)  =
    \Omega_{L_{d}}^{C_{d,r}},    
  \end{equation*}
  so that $F_{L_{d}}|_{C_{d,r}}$ is a symplectomorphism.
\end{remark}


\section{Future work}
\label{Sect:future}

It is well known that systems with group symmetry can be reduced and
the resulting systems provide a useful way to understand the ``core''
dynamics and, in some cases, a practical way of solving their
equations of motion.  Therefore, it is a very natural continuation of
the current work to introduce a notion of DSOCS with symmetry group
and develop a reduction procedure for these systems. We intend to
tackle this problem following the approach to reduce discrete
nonholonomic systems used
in~\cite{ar:fernandez_tori_zuccalli-lagrangian_reduction_of_discrete_mechanical_systems}.

Given a numerical integrator of a continuous system, it is very
important to know how well it approximates the actual solution of the
original system. In the unconstrained case, such analysis can be
performed as follows. As a first step, an \emph{exact discrete
  Lagrangian} is defined: it has the property that its discrete
trajectories coincide with the trajectories of the original system (at
specific discrete times). Except in a few trivial cases, such exact
Lagrangians cannot be constructed explicitly, so a second step is to
construct discrete Lagrangians that approximate the exact one. Using
this approach, it is possible to give estimates of the goodness of the
numerical integrator (see~\cite{ar:marsden_west-discrete_mechanics_and_variational_integrators}, Part 2). For discrete
systems with nonholonomic constraints the same type of error analysis
was started
in~\cite{ar:deleon_martindediego_santameriamerino-geometric_integrators_and_nonholonomic_mechanics}. However,
their work still needs to be completed after the results
of~\cite{ar:patrick_cuell-eror_analysis_of_variational_integrators_of_unconstrained_lagrangian_systems}. Perhaps,
this could be done by giving an adequate extension
of~\cite{ar:patrick_cuell-eror_analysis_of_variational_integrators_of_unconstrained_lagrangian_systems}
to the nonholonomic case. Even more, we would like to extend the whole
program to the error analysis of DSOCSs.



\begin{thebibliography}{10}

\bibitem{ar:benito_deleon_martindediego-higher_order_discrete_lagrangian_mechanics}
  R. Benito, M. de~Le{\'o}n and D.~Mart{\'{\i}}n~de Diego,
  \emph{Higher-order discrete {L}agrangian mechanics}, Int. J. Geom. Methods
  Mod. Phys. \textbf{3} (2006), no.~3, 421--436. \MR{2232878 (2007e:70015)}

\bibitem{ar:benito_martindediego-hidden_simplecticity_in_hamilton_s_principle_algorithm}
  R. Benito and D. Mart{\'{\i}}n~de Diego, \emph{Hidden symplecticity
    in {H}amilton's principle algorithms}, Differential geometry and
  its applications, Matfyzpress, Prague, 2005,
  pp.~411--419. \MR{2268952 (2007k:70028)}

\bibitem{bo:bloch-nonholonomic_mechanics_and_control} A.~M. Bloch,
  \emph{Nonholonomic mechanics and control}, Interdisciplinary Applied
  Mathematics, vol.~24, Springer-Verlag, New York, 2003. \MR{1978379
    (2004e:37099)}

\bibitem{th:borda-sistemas_mecanicos_discretos_con_vinculos_de_orden_2}
  N. Borda, \emph{Sistemas mec{\'a}nicos discretos con v{\'\i}nculos
    de orden 2}, Tesis de Maestr{\'\i}a en Ciencias F{\'\i}sicas,,
  Instituto Balseiro, Bariloche, Argentina, 2011.

\bibitem{ar:campos_cendra_diaz_martin-discrete_lagrange_dalembert_poincare_equations_for_euler's_disk}
  C. Campos, H. Cendra, V. D{\'{\i}}az and D. Mart{\'{\i}}n~de Diego,
  \emph{Discrete {L}agrange-d'{A}lembert-{P}oincar\'e equations for
    {E}uler's disk}, Rev. R.  Acad. Cienc. Exactas F\'\i
  s. Nat. Ser. A Math. RACSAM \textbf{106} (2012), no.~1,
  225--234. \MR{2892145}

\bibitem{ar:cendra_grillo-generalized_nonholonomic_mechanics_servomechanisms_and_related_brackets}
  H.~Cendra and S.~Grillo, \emph{Generalized nonholonomic mechanics,
  servomechanisms and related brackets}, J. Math. Phys. \textbf{47} (2006),
  no.~2, 022902, 29. \MR{MR2208156 (2007a:70019)}

\bibitem{ar:cendra_grillo-lagrangian_systems_with_higher_order_constraints}
H.~Cendra and S.~ Grillo, \emph{Lagrangian systems with higher order
  constraints}, J. Math. Phys. \textbf{48} (2007), no.~5, 052904, 35.
  \MR{2329856 (2008e:70017)}

\bibitem{ar:cendra_ibort_de_leon_martin-a_generalization_of_chetaevs_principle_for_a_class_of_higher_order_nonholonomic_constraints}
  H.~Cendra, A.~Ibort, M.~de~Le{\'o}n and D.~Mart{\'{\i}}n~de Diego,
  \emph{A generalization of {C}hetaev's principle for a class of
    higher order nonholonomic constraints}, J. Math. Phys. \textbf{45}
  (2004), no.~7, 2785--2801. \MR{MR2067586 (2005d:70004)}

\bibitem{bo:cendra_marsden_ratiu-lagrangian_reduction_by_stages}
  H.~Cendra, J.~Marsden, and T.~Ratiu, \emph{Lagrangian reduction by
    stages}, Mem. Amer. Math. Soc. \textbf{152} (2001), no.~722,
  x+108. \MR{MR1840979 (2002c:37081)}

\bibitem{ar:chetaev-on_the_gauss_principle}
  N.~G. Chetaev, \emph{On the gauss principle}, Izv. Fiz-Mat. Obsc. Kazan Univ.
  \textbf{7} (1934), 68--71.

\bibitem{ar:cortes_martinez-non_holonomic_integrators}
  J.~Cort{\'e}s and S.~Mart{\'{\i}}nez, \emph{Non-holonomic integrators},
  Nonlinearity \textbf{14} (2001), no.~5, 1365--1392. \MR{MR1862825
  (2002h:37165)}

\bibitem{bo:cortes-non_holonomic}
  J. Cort{\'e}s, \emph{Geometric, control and numerical aspects of
  nonholonomic systems}, Lecture Notes in Mathematics, vol. 1793,
  Springer-Verlag, Berlin, 2002. \MR{MR1942617 (2003k:70013)}

\bibitem{ar:crampin_sarlet_cantrijn-higher_order_differential_equations_and_higher_order_lagrangian_mechanics}
  M.~Crampin, W.~Sarlet and F.~Cantrijn, \emph{Higher-order differential
  equations and higher-order {L}agrangian mechanics}, Math. Proc. Cambridge
  Philos. Soc. \textbf{99} (1986), no.~3, 565--587. \MR{830369 (87i:58053)}

\bibitem{ar:deleon_martindediego_santameriamerino-geometric_integrators_and_nonholonomic_mechanics}
M.~de~Le{\'o}n, D.~Mart{\'{\i}}n~de Diego and A.~Santamar{\'{\i}}a-Merino,
  \emph{Geometric integrators and nonholonomic mechanics}, J. Math. Phys.
  \textbf{45} (2004), no.~3, 1042--1064. \MR{2036181 (2004j:37157)}

\bibitem{bo:deleon_rodrigues-generalized_classical_mechanics_and_field_theory}
  M.~de~Le{\'o}n and P.~Rodrigues, \emph{Generalized classical
    mechanics and field theory}, North-Holland Mathematics Studies,
  vol. 112, North-Holland Publishing Co., Amsterdam, 1985, A
  geometrical approach of Lagrangian and Hamiltonian formalisms
  involving higher order derivatives, Notes on Pure Mathematics,
  102. \MR{808964 (87m:58059)}

\bibitem{bo:dobronravov-the_fundamentals_of_mechanics_of_nonholonomic_systems}
V.~Dobronravov, \emph{The fundamentals of the mechanics of nonholonomic
  systems}, Vysshaya Shkola, 1970.

\bibitem{ar:fernandez_tori_zuccalli-lagrangian_reduction_of_discrete_mechanical_systems}
  J.~ Fern{\'a}ndez, C.~Tori and M.~ Zuccalli, \emph{Lagrangian
    reduction of nonholonomic discrete mechanical systems},
  J. Geom. Mech.  \textbf{2} (2010), no.~1, 69--111, Also,
  \href{http://arXiv.org/abs/1004.4288}{{\tt
      arXiv:1004.4288}}. \MR{2646536}

\bibitem{ar:grillo_maciel_perez-closed_loop_and_constrained_mechanical_systems}
  S.~Grillo, F.~Maciel and D.~P{\'e}rez, \emph{Closed-loop and constrained
  mechanical systems}, Int. J. Geom. Methods Mod. Phys. \textbf{7} (2010),
  no.~5, 857--886. \MR{2720548 (2011j:70027)}

\bibitem{th:grillo-sistemas_noholonomos_generalizados} S.~ Grillo,
  \emph{Sistemas nohol\'onomos generalizados}, Tesis de Doctorado en
  Matem\'atica, Universidad Nacional del Sur, Bah{\'\i}a Blanca, 2007.

\bibitem{ar:grillo_marsden_nair-lyapunov_constraints_and_global_asymptotic_stabilization}
  S.~Grillo, J.~ Marsden, and S.~Nair, \emph{Lyapunov constraints and
    global asymptotic stabilization}, J. Geom. Mech. \textbf{3}
  (2011), no.~2, 145--196. \MR{2824611 (2012m:70030)}

\bibitem{ar:grillo-higher_order_constrained_hamiltonian_systems} S.~
  Grillo, \emph{Higher order constrained {H}amiltonian systems}, J.
  Math. Phys. \textbf{50} (2009), no.~8, 082901, 34. \MR{2554421
    (2010i:70015)}

\bibitem{bo:hairer_lubich_wanner-geometric_numerical_integration}
  E.~Hairer, C.~Lubich, and G.~Wanner, \emph{Geometric numerical
    integration}, second ed., Springer Series in Computational
  Mathematics, vol.~31, Springer-Verlag, Berlin, 2006. \MR{MR2221614
    (2006m:65006)}

\bibitem{ar:krupkova-higher_order_mechanical_systems_with_constraints}
  O.~Krupkov{\'a}, \emph{Higher-order mechanical systems with
    constraints}, J.  Math. Phys. \textbf{41} (2000), no.~8,
  5304--5324. \MR{1770957 (2001i:37097)}

\bibitem{ar:colombo_martindediego_zuccalli-higher_order_discrete_variational_problems_with_constraints}
  L.~Colombo, D.~Mart{\'\i}n de~Diego and M.~Zuccalli,
  \emph{Higher-order discrete variational problems with constraints},
  J. Math. Phys. \textbf{54} (2013).

\bibitem{ar:marle-kinematic_and_geometric_constraints_servomechanism_and_control_of_mechanical_systems}
C.-M. Marle, \emph{Kinematic and geometric constraints, servomechanism and
  control of mechanical systems}, Rend. Sem. Mat. Univ. Politec. Torino
  \textbf{54} (1996), no.~4, 353--364, Geometrical structures for physical
  theories, II (Vietri, 1996). \MR{MR1618126 (99e:70040)}

\bibitem{ar:marsden_west-discrete_mechanics_and_variational_integrators}
  J.~Marsden and M.~West, \emph{Discrete mechanics and variational
  integrators}, Acta Numer. \textbf{10} (2001), 357--514. \MR{MR2009697
  (2004h:37130)}

\bibitem{ar:mclachlan_perlmutter-integrators_for_nonholonomic_mechanical_systems}
  R.~McLachlan and M.~Perlmutter, \emph{Integrators for nonholonomic mechanical
  systems}, J. Nonlinear Sci. \textbf{16} (2006), no.~4, 283--328.
  \MR{MR2254707 (2008d:37154)}

\bibitem{bo:neimark_fufaev-dynamics_of_nonholonomic_systems}
  Yu.~Ne{\u\i}mark and N.~Fufaev, \emph{Dynamics of nonholonomic
    systems}, Translations of Mathematical Monographs, vol.~33,
  American Mathematical Society, Providence, RI, 1972.

\bibitem{ar:patrick_cuell-eror_analysis_of_variational_integrators_of_unconstrained_lagrangian_systems}
  G.~Patrick and C.~Cuell, \emph{Error analysis of variational
    integrators of unconstrained {L}agrangian systems},
  Numer. Math. \textbf{113} (2009), no.~2, 243--264. \MR{2529508
    (2010f:37146)}

\bibitem{ar:shiriaev_perram_canudasdewit-constructive_tool_for_orbital_stabilization_of_underactuated_nonlinear_systems}
  A.~Shiriaev, J.~Perram, and C.~Canudas-de Wit, \emph{Constructive
    tool for orbital stabilization of underactuated nonlinear systems:
    virtual constraints approach}, IEEE Trans. Automat. Control
  \textbf{50} (2005), no.~8, 1164--1176. \MR{2156044 (2006k:93113)}

\bibitem{ar:swaczyna-mechanical_systems_with_nonholonomic_constraints_of_the_second_order}
M.~Swaczyna, \emph{Mechanical systems with nonholonomic constraints of the
  second order}, AIP Conf. Proc., vol. 1360, 2011, pp.~164--169.

\end{thebibliography}


\def\cprime{$'$} \def\polhk#1{\setbox0=\hbox{#1}{\ooalign{\hidewidth
  \lower1.5ex\hbox{`}\hidewidth\crcr\unhbox0}}} \def\cprime{$'$}
  \def\cprime{$'$}
\providecommand{\bysame}{\leavevmode\hbox to3em{\hrulefill}\thinspace}
\providecommand{\MR}{\relax\ifhmode\unskip\space\fi MR }
\providecommand{\MRhref}[2]{%
  \href{http://www.ams.org/mathscinet-getitem?mr=#1}{#2}
}
\providecommand{\href}[2]{#2}


\end{document}